\documentclass[12pt]{amsart}

\usepackage[english]{babel}
\usepackage{mathrsfs,amssymb}
\usepackage{mathtools}
\usepackage[colorlinks, citecolor = blue]{hyperref}

\usepackage[shortlabels]{enumitem}
\setlist[itemize]{leftmargin=25pt}
\setlist[enumerate]{leftmargin=25pt}

\newtheorem{theorem}{Theorem}[section]

\newtheorem{prop}[theorem]{Proposition}
\newtheorem{cor}[theorem]{Corollary}

\theoremstyle{definition}
\newtheorem{definition}[theorem]{Definition}

\theoremstyle{remark}
\newtheorem{remark}[theorem]{Remark}

\newtheorem{example}[theorem]{Example}

\numberwithin{equation}{section}

\DeclareMathOperator*{\esssup}{ess\,sup}

\DeclareMathOperator{\osc}{osc}
\DeclareMathOperator{\loc}{loc}
\DeclareMathOperator{\UMD}{UMD}
\DeclareMathOperator{\mart}{mart}

\newcommand{\N}{\ensuremath{\mathbb{N}}}

\newcommand{\R}{\ensuremath{\mathbb{R}}}

\renewcommand{\P}{\ensuremath{\mathbb{P}}}

\newcommand{\mc}{\mathcal}
\newcommand{\ms}{\mathscr}

\newcommand{\mbs}{\boldsymbol}

\DeclarePairedDelimiter\abs{\lvert}{\rvert}

\DeclarePairedDelimiter\cbrace\{\}
\DeclarePairedDelimiter\ha()
\DeclarePairedDelimiter{\ip}\langle\rangle
\DeclarePairedDelimiter{\nrm}\lVert\rVert

\newcommand{\nrmb}[1]{\bigl\|#1\bigr\|}
\newcommand{\absb}[1]{\bigl|#1\bigr|}
\newcommand{\hab}[1]{\bigl(#1\bigr)}
\newcommand{\cbraceb}[1]{\bigl\{#1\bigr\}}
\newcommand{\ipb}[1]{\bigl\langle#1\bigr\rangle}

\newcommand{\nrms}[1]{\Bigl\|#1\Bigr\|}

\newcommand{\has}[1]{\Bigl(#1\Bigr)}

\newcommand{\dd}{\hspace{2pt}\mathrm{d}}
\newcommand{\ddn}{\mathrm{d}}

\let \la=\lambda
\let \e=\varepsilon
\let \d=\delta
\let \o=\omega
\let \a=\alpha

\let \O=\Omega

\let \G=\Gamma
\let \ga=\gamma

\allowdisplaybreaks

\begin{document}

\title[Operator-free sparse domination]
{Operator-free sparse domination}

\author[A.K. Lerner]{Andrei K. Lerner}
\address[A.K. Lerner]{Department of Mathematics,
Bar-Ilan University, 5290002 Ramat Gan, Israel}
\email{lernera@math.biu.ac.il}

\author[E. Lorist]{Emiel Lorist}
\address[E. Lorist]{Department of Mathematics and Statistics\\ University of Helsinki \\ P.O. Box 68\\
FI-00014 Helsinki\\ Finland}
\email{emiellorist@gmail.com}

\author[S. Ombrosi]{Sheldy Ombrosi}
\address[S. Ombrosi]{Departamento de Matem\'atica\\
Universidad Nacional del Sur\\
Bah\'ia Blanca, 8000, Argentina}\email{sombrosi@uns.edu.ar}

\thanks{The second author was supported by the Academy of Finland through grant no. 336323. The third author was partially supported
by ANPCyT PICT 2018-2501.}

\begin{abstract}
We obtain a sparse domination principle for an arbitrary family of functions $f(x,Q)$, where $x\in {\mathbb R}^n$ and $Q$ is a cube in~${\mathbb R}^n$.
When applied to operators, this result recovers our recent works \cite{LO19,Lo19b}. On the other hand, our sparse domination principle can be also applied to non-operator objects.
In particular, we show applications to generalized Poincar\'e--Sobolev inequalities, tent spaces, and general dyadic sums.
Moreover, the flexibility of our result allows us to treat operators that are not localizable in the sense of  \cite{Lo19b},
as we will demonstrate in an application to vector-valued square functions.
\end{abstract}



\maketitle

\section{Introduction}
Sparse domination is a recent technique allowing one to estimate (in norm, pointwise or dually) many operators in harmonic analysis by simple expressions of the form
$$\sum_{Q\in {\mathcal S}}\langle f\rangle_{p,Q}\chi_Q,$$
where $\langle f\rangle_{p,Q}=\big(\frac{1}{|Q|}\int_Q|f|^p\big)^{1/p}$ for $p \in (0,\infty)$ and ${\mathcal S}$ is a sparse family of cubes in ${\mathbb R}^n$.

Primarily motivated by sharp quantitative weighted norm inequalities, sparse domination has quickly transformed into a very active area, dealing with
various operators within and beyond the Calder\'on--Zygmund theory. During the last five years a number of sparse domination principles (that is, general results establishing
sparse domination for a given class of operators) have appeared e.g. in the works \cite{BBR20, BRS20,BM21, BFP16,CCDO17,CDPV20, Le16, Le19, LO19,Lo19b}.

Let us consider a particular line of research in this direction, for which the starting point is the so-called local mean oscillation
estimate (see \cite{Hy14b,Le10})
\begin{equation}\label{lmo}
|f-m_f(Q)|\chi_Q\lesssim \sum_{P\in {\mathcal S}}\o_{\la}(f;P)\chi_P,
\end{equation}
where $f$ is an arbitrary measurable function and $m_f(Q)$ and $\o_{\la}(f;Q)$ denote a median value and the local mean oscillation of $f$ on the cube~$Q$, respectively.

The local mean oscillation estimate can be regarded as the first operator-free sparse domination result, but its main application was to operators. Specifically,
this estimate was applied by the first author in~\cite{Le13a} to a Calder\'on--Zygmund operator $T$, using $Tf$ instead of~$f$ in \eqref{lmo}. This provided
norm sparse domination for $T$ and, as a result, an alternative proof of the $A_2$-theorem, which was  first proven by Hyt\"onen~\cite{Hy12}. Later, this norm sparse domination result was improved to pointwise
sparse domination simultaneously by Conde-Alonso--Rey and the first author and Nazarov in \cite{CR16,LN15}.

The methods in \cite{CR16,LN15} still depended on (\ref{lmo}). The drawback of this approach is that it necessitates estimating local mean oscillations of~$T$, although
$T$ is not a well-localized operator. For this reason, the results in \cite{CR16,LN15} hold under the $\log$-Dini assumption on the kernel of the Calder\'on--Zygmund operator $T$.

The next step was taken by Lacey in \cite{La17b}, where pointwise sparse domination for $T$ was obtained under the usual Dini assumption on the kernel of $T$.
The main new realisation in \cite{La17b} was that it suffices to estimate suitable truncations of $T$, which can be done without the use of (\ref{lmo}).
The proof of the pointwise sparse domination result for $T$ was subsequently simplified by the first author in \cite{Le16} and the first and third authors in \cite{LO19}, in which a general sparse domination principle was
established, allowing one to deal with a vast number of ``smooth'' operators. The main result of \cite{LO19} was then extended by the second author \cite{Lo19b} into several directions, including the setting of vector-valued functions on spaces of homogeneous type,
along with the concept of $\ell^r$-sparse domination.

The development we have so far described can be summarised in the following diagram:
$$
(\ref{lmo})\rightarrow\text{\cite{Le13a}}\rightarrow\text{\cite{CR16,LN15}}\rightarrow\text{\cite{La17b}}\rightarrow\text{\cite{Le16}}\rightarrow\text{\cite{LO19}}\rightarrow\text{\cite{Lo19b}}.
$$
Starting from \cite{La17b}, the local mean oscillation estimate (\ref{lmo}) has not played a role in the obtained sparse domination results. Therefore, this development can be viewed as an evolution
from sparse domination for arbitrary functions (expressed in (\ref{lmo})) to sparse domination for operators.

In the present article,  we return sparse domination to its roots, using functions rather than operators.  We will essentially use the techniques developed in \cite{LO19,Lo19b}. Our key novel point is the language in which our main results are written.
This language unifies (\ref{lmo}) with all the results containing in \cite{LO19,Lo19b}. More important, it allows us to deal with many non-operator objects, which have not yet been investigated using sparse domination techniques. This development can be seen in analogy to the development of Rubio de Francia extrapolation, which was first proven for operators, but was later realised to be much more versatile and applicable in a formulation for functions (see \cite{CMP11} for a discussion on this development).

\bigskip

Let us give a flavour of our language.
Precise definitions and statements will be given in subsequent sections.
We shall deal with functions $f(x,Q)$ of two ``variables'', being the points $x\in {\mathbb R}^n$ and the cubes $Q\subset {\mathbb R}^n$. For brevity we denote such functions by $f_Q(x)$. Observe
that the role of the cubes $Q$ may vary. For example, one can define $f_Q(x):=T(f\chi_{\a Q})(x)$, where $T$ is a given operator. On the other hand, in the theory of tent spaces,
a typical definition will be $f_Q(x):=\int_{\G_{\ell_Q}(x)}f(y,t)\dd\mu(y,t)$, where $\G_{\ell_Q}(x)$ is a cone in~${\mathbb R}^{n+1}_+$, truncated according to the side length of $Q$.

Our main result, Theorem~\ref{ofsdp}, provides pointwise sparse domination for $|f_Q|$ for a fixed cube $Q \subset \R^n$. The dominating sparse object is expressed in terms of $f_P$ for dyadic subcubes $P$ of $Q$ and a certain family of functions $f_{P,Q}$, which connects $f_P$ and $f_Q$ in a natural way. A typical example is $f_{P,Q}:=f_Q-f_P$, but, depending on the context, one can make more clever choices of $f_{P,Q}$.

\bigskip

The article is organised as follows. Section \ref{sec:maindefinitions} contains the main definitions. In particular, our language is introduced there. Section \ref{sec:mainresults} contains our main results, which are pointwise and bilinear form
operator-free sparse domination principles, expressed in Theorems \ref{ofsdp} and \ref{bf}. In Section \ref{sec:oldresults} we show that our new results contain a vast number of
previously known sparse domination results as particular cases.

Sections \ref{sec:poincare}-\ref{sec:dyadic} present new applications. Section \ref{sec:poincare} is devoted to generalized Poincar\'e--Sobolev inequalities. In a recent work on this topic \cite{PR19},
P\'{e}rez and Rela obtained a weighted self-improving result, assuming the $A_{\infty}$-condition on the weight. They asked whether the $A_{\infty}$-condition can be removed.
Using our sparse domination principle, we give an affirmative answer to this question.

In Section \ref{sec:tent}, we give a ``sparse'' proof of the theorem of Coifman--Meyer--Stein \cite{CMS85} on  the main relation between two basic operators in the
theory of tent spaces. In particular, this improves the good-$\lambda$ inequality for these operators established in \cite{CMS85}. We also obtain similar
results for vector-valued tent spaces.

In Section \ref{sec:squarefun} we give a simplified proof of a recent result by Xu \cite{Xu21} about sharp bounds for the vector-valued vertical square function, which was an important ingredient in his answer to a question by Naor and Young \cite{NY20} about sharp bounds for the heat semigroup on~${\mathbb R}^n$.

In Section \ref{sec:dyadic} we obtain a simple sufficient condition allowing one, when dealing with the dyadic sums $\sum\a_Q\chi_Q$, to replace the summation over all dyadic
subcubes of a given cube by the summation over a sparse family.
As an application, we will generalise and provide a new proof of a result by Honz\'ik and Jaye \cite{HJ12}
on a sharp good-$\lambda$ inequality for the nonlinear dyadic potential.

\bigskip

Throughout the article we use the notation $A\lesssim B$ if $A\le CB$ with some independent constant $C$. We write $A\eqsim B$ if $A\lesssim B$ and $B\lesssim A$.

\section{Main definitions}\label{sec:maindefinitions}
\subsection{Dyadic cubes} Denote by ${\mathcal Q}$ the set of all cubes $Q\subset {\mathbb R}^n$ with sides parallel to the axes. Given a cube $Q\in {\mathcal Q}$, denote by
${\mathcal D}(Q)$ the set of all dyadic cubes with respect to $Q$, that is, the cubes obtained by repeated subdivision of $Q$ and each of its descendants into $2^n$ congruent subcubes.

Following \cite[Def. 2.1]{LN15}, a dyadic lattice ${\mathscr D}$ in ${\mathbb R}^n$ is any collection of cubes such that
\begin{enumerate}
\renewcommand{\labelenumi}{(\roman{enumi})}
\item Any child of $Q\in{\mathscr D}$ is in ${\mathscr D}$ as well, i.e. $\mc{D}(Q) \subseteq \ms{D}$.
\item
Any $Q',Q''\in {\mathscr D}$ have a common ancestor, i.e. there exists a $Q\in{\mathscr D}$ such that $Q',Q''\in {\mathcal D}(Q)$.
\item
For every compact set $K\subset {\mathbb R}^n$, there exists a cube $Q\in {\mathscr D}$ containing $K$.
\end{enumerate}
Let $Q\in {\mathcal Q}$. We say that a family of dyadic cubes ${\mathcal F}\subset {\mathcal D}(Q)$ is \emph{contracting} if
${\mathcal F}=\cup_{k=0}^{\infty}{\mathcal F}_k$, where $\mc{F}_0 = \cbrace{Q}$,
 each ${\mathcal F}_k$ is a family of pairwise disjoint cubes, and for $\O_k=\cup_{P\in {\mathcal F}_k}P$ we have
$\O_{k+1}\subset \O_k$ and $|\O_k|\to 0$ as $k\to \infty$.
Given a contracting family ${\mathcal F}\subset {\mathcal D}(Q)$, for $P\in {\mathcal F}_k$ with $k\in \N \cup \cbrace{0},$ we denote
$$E_P:=P\setminus \bigcup_{P'\in {\mathcal F}_{k+1}}P' = P\setminus \Omega_{k+1}.$$
Observe that the sets $\{E_P\}_{P\in {\mathcal F}}$ are pairwise disjoint.

\begin{definition} Let $\eta \in (0,1)$ and $Q \in \mc{Q}$. We say that a family ${\mathcal F}\subset {\mathcal D}(Q)$ is $\eta$-sparse if it is contracting and
$|E_P|\ge \eta|P|$ for all $P\in {\mathcal F}$.
\end{definition}

Note that our definition of a sparse family of cubes is slightly more restrictive than the usual definition in the literature. In particular, we assume a sparse family of cubes to have exactly one maximal cube, the cube $Q$.

\subsection{The \texorpdfstring{$\ell^r$}{lr}-condition} Consider a family of measurable functions $\{f_Q, f_{P,Q}\}\colon{\mathbb R}^n\to {\mathbb R}$, where $Q\in {\mathcal Q}$, $P\in {\mathcal D}(Q)$. We note that our choice of $\mathbb{R}$ as the scalar field is inconsequential, i.e. all subsequent results hold for complex-valued functions as well.

 We introduce  a compatibility condition on such a family of functions, which is implicitly contained in \cite{Lo19b}. We will elaborate on the connection to~\cite{Lo19b} in Subsection \ref{subsection:ellr}.

\begin{definition}\label{r-cond}
Let $r \in (0,\infty)$. We say that the family
$$\{f_Q,f_{P,Q}\}_{Q\in {\mathcal Q}, P\in {\mathcal D}(Q)}$$
satisfies the \emph{$\ell^r$-condition}
if there exists a $C_r>0$ such that for every $Q\in {\mathcal Q}$ and every $P_1,\dots, P_m\in {\mathcal D}(Q)$ with $P_m\subset\dots\subset P_1$, we have for a.e. $x \in P_m$,
$$|f_{P_1}(x)|\le C_r\has{\sum_{k=1}^{m-1}|f_{P_{k+1},P_{k}}(x)|^r+|f_{P_m}(x)|^r}^{1/r}.$$
\end{definition}

Given an arbitrary family of measurable functions $f_Q\colon {\mathbb R}^n\to {\mathbb R}$, a canonical choice for $f_{P,Q}$ is given by
\begin{equation}\label{eq:mainfPQ}
  f_{P,Q}:=f_Q-f_P,
\end{equation}
for which the $\ell^r$-condition holds trivially with $C_r=1$ for $0<r\le 1$.

\subsection{Sharp maximal operators} For a measurable function $f \colon \R^n \to \R$ we define its \emph{standard oscillation} over a cube $Q \in \mc{Q}$ by
$$\osc(f;Q)=\esssup_{x',x''\in Q}|f(x')-f(x'')|.$$
Furthermore, for $q \in (0,\infty)$ we define its \emph{$q$-oscillation} over a cube $Q \in \mc{Q}$ by
$$\osc_q(f;Q)=\has{\frac{1}{\abs{Q}^2}\int_{Q\times Q}|f(x')-f(x'')|^q\dd x'\dd x''}^{1/q}.$$

Using these oscillations, we can now define sharp maximal operators associated to a family $\{f_{P,Q}\}$, of which precursors can be found in \cite{LO19, Lo19b}.
\begin{definition}\label{sharp} Given a family $\{f_{P,Q}\}_{Q\in {\mathcal Q}, P\in {\mathcal D}(Q)}$,
define the \emph{sharp maximal function} $m_Q^{\#}f$ for $Q \in \mc{Q}$ by
$$m_Q^{\#}f(x)=\sup_{P\in {\mathcal D}(Q):x \in P}\osc(f_{P,Q};P), \qquad x \in Q,$$
and for $q \in (0,\infty)$ define the \emph{sharp $q$-maximal function} by
$$m_{Q,q}^{\#}f(x)=\sup_{P\in {\mathcal D}(Q):x \in P}\osc_q(f_{P,Q};P)), \qquad x \in Q.$$
\end{definition}

\subsection{Nonincreasing rearrangements} The non-increasing rearrangement of a measurable function $f \colon \R^n \to\R$ is defined by
$$f^*(t)=\inf\big\{\a>0:|\{x\in {\mathbb R}^n:|f(x)|>\a\}|\le t\big\},\qquad t \in \R_+.$$
Observe that $(|f|^{\d})^*(t)=f^*(t)^{\d}$ for every $\d>0$. This, along with Chebyshev's inequality $f^*(t)\le \frac{1}{t}\|f\|_{L^1(\R^n)}$, implies that
\begin{equation}\label{rearch}
f^*(t)\le \frac{1}{t^{\frac1\d}}\|f\|_{L^{\d}(\R^n)}.
\end{equation}
We also have
\begin{equation}\label{rear}
\absb{\{x\in {\mathbb R}^n:|f(x)|>f^*(t)\}}\le t,
\end{equation}
which is a consequence of the fact that the distribution function is right-continuous.

%
%
%

\section{Main results}\label{sec:mainresults}
\subsection{A toy domination principle} We start our analysis by observing that the $\ell^r$-condition allows us to bound $f_Q$ for every $Q \in \mc{Q}$ by a sum over an arbitrary contracting family of cubes. In particular we note that we do not need
a sparseness assumption in the following statement.

\begin{prop}\label{tsd}
Let $r\in (0,\infty)$ and let $\{f_Q,f_{P,Q}\}_{Q \in \mc{Q}, P \in \mc{D}(Q)}$ satisfy the $\ell^r$-condition. Let $Q\in {\mathcal Q}$ and
let ${\mathcal F}\subset {\mathcal D}(Q)$ be a contracting family of cubes. Then for a.e. $x\in Q$,
$$|f_Q(x)|^r\le C_r^r\sum_{k=0}^{\infty}\sum_{P\in {\mathcal F}_k}\Big(|f_P(x)|^r\chi_{E_P}(x)
+\sum_{P'\in {\mathcal F}_{k+1}:P'\subset P}|f_{P',P}(x)|^r\chi_{P'}(x)\Big).
$$
\end{prop}

\begin{proof}
Since $|\O_k|\to 0$, there is a set $N\subset Q$ of measure $0$ such that, for every $x\in Q\setminus N$, there are only finitely many $k\in {\mathbb N}$ with $x\in \O_k$.

Fix $x\in Q\setminus N$. Then there exist $P_k\in {\mathcal F}_k$ for $k=0,\dots, m$ such that
$$x\in P_{m}\subset P_{m-1}\subset\dots\subset P_0=Q$$
and $x\not\in \O_{m+1}$. Hence, by the $\ell^r$-condition,
\begin{align*}
|f_Q(x)|^r&\leq C_r^r\has{|f_{P_m}(x)|^r+\sum_{k=0}^{m-1}|f_{P_{k+1},P_{k}}(x)|^r}\\
&=C_r^r\has{|f_{P_m}(x)|^r\chi_{E_{P_m}}(x)+\sum_{k=0}^{m-1}|f_{P_{k+1},P_{k}}(x)|^r\chi_{P_{k+1}}(x)}.
\end{align*}
In order to make this expression independent of the particular choice of $P_k$, we add zero terms. This allows us to write
$$|f_{P_m}(x)|^r\chi_{E_{P_m}}(x)=\sum_{k=0}^{\infty}\sum_{P\in {\mathcal F}_k}|f_P(x)|^r\chi_{E_P}(x)$$
and
$$\sum_{k=0}^{m-1}|f_{P_{k+1},P_{k}}(x)|^r\chi_{P_{k+1}}(x)=\sum_{k=0}^{\infty}\sum_{P\in {\mathcal F}_k}\sum_{P'\in {\mathcal F}_{k+1}:P'\subset P}|f_{P',P}(x)|^r\chi_{P'}(x),$$
which completes the proof.
\end{proof}

\subsection{A pointwise sparse domination principle}
In order to estimate  the terms $f_{P',P}$ in Proposition \ref{tsd} effectively, we make an additional assumption on the family $\{f_{P,Q}\}_{Q\in {\mathcal Q}, P\in {\mathcal D}(Q)}$. Indeed, we will assume for $Q \in \mc{Q}$ and $P \in \mc{D}(Q)$ that
\begin{equation}\label{cond}
|f_{P,Q}|\le |f_P|+|f_Q|.
\end{equation}
Observe that this assumption is not really restrictive. In particular, for the main example in \eqref{eq:mainfPQ}, (\ref{cond}) holds trivially.

Our first main result is the following.

\begin{theorem}\label{ofsdp} Let $r \in (0,\infty)$ and let $\{f_Q, f_{P,Q}\}_{Q\in {\mathcal Q}, P\in {\mathcal D}(Q)}$ satisfy the $\ell^r$-condition and (\ref{cond}).
For any $Q \in \mc{Q}$ and $\eta \in (0,1)$
there exists an $\eta$-sparse family ${\mathcal F}\subset \mc{D}(Q)$ such that for a.e. $x\in Q$,
$$|f_Q(x)|\lesssim C_r\Big(\sum_{P\in {\mathcal F}}\ga_P^r\chi_P(x)\Big)^{1/r},$$
where
$$\ga_P:=(f_{P}\chi_{P})^*\hab{|P|\tfrac{1-\eta}{2^{n+2}}}+(m^{\#}_Pf)^*\hab{|P|\tfrac{1-\eta}{2^{n+2}}}.$$
\end{theorem}

\begin{proof}
We construct the family ${\mathcal F}\subset\mc{D}(Q)$ inductively. Set ${\mathcal F}_0=\{Q\}$. Next, given a collection of pairwise disjoint cubes ${\mathcal F}_k$, let
us describe how to construct ${\mathcal F}_{k+1}$.

Fix a cube $P\in {\mathcal F}_k$. Consider the sets
\begin{align*}
  \O_1(P)&:=\cbraceb{x\in P:|f_P(x)|>(f_P\chi_P)^*\hab{|P|\tfrac{1-\eta}{2^{n+2}}}},\\
  \O_2(P)&:=\cbraceb{x\in P:m_P^{\#}f(x)>\big(m_P^{\#}f\big)^*\hab{|P|\tfrac{1-\eta}{2^{n+2}}}},
\end{align*}
and denote $\O(P):=\O_1(P)\cup\O_2(P)$. Then, by (\ref{rear}), we know that $|\O(P)|\le \frac{1-\eta}{2^{n+1}}|P|$.

Apply the local Calder\'on--Zygmund decomposition to $\chi_{\O(P)}$ at height $\frac{1}{2^{n+1}}$. We obtain a family of
pairwise disjoint cubes $\mc{S}_P=\{P_j\}$, dyadic with respect to $P$, such that for
$N_P=\Omega(P)\setminus\cup_jP_j$ we have that $|N_P|=0$ and for every $P_j\in \mc{S}_P$,
\begin{equation}\label{cz}
\frac{1}{2^{n+1}}|P_j|\le |P_j\cap\Omega(P)|\le \frac{1}{2}|P_j|.
\end{equation}
In particular, it follows from this that
\begin{equation}\label{half}
\sum_j|P_j|\le 2^{n+1}|\O(P)|\le (1-\eta) |P|.
\end{equation}
We define  ${\mathcal F}_{k+1}=\cup_{P\in {\mathcal F}_k}\mc{S}_P$.
Setting ${\mathcal F}=\cup_{k=0}^{\infty}{\mathcal F}_k$, we note  by
\eqref{half} that ${\mathcal F}$ is $\eta$-sparse.

Let us now prepare to apply Proposition \ref{tsd} with the constructed family ${\mathcal F}$. Fix $P\in {\mathcal F}_k$ for some $k \in \N \cup \cbrace{0}$.
Since $\abs{N_P}=0$, almost every point of the set $\O_1(P)$ is covered by a cube $P'\in {\mathcal F}_{k+1}$.
Therefore, we have
\begin{equation}\label{Ep}
|f_P(x)|\le (f_P\chi_P)^*\hab{|P|\tfrac{1-\eta}{2^{n+2}}}, \qquad x \in E_P.
\end{equation}

It remains to estimate $|f_{P',P}|\chi_{P'}$ for
$P\in {\mathcal F}_k$ and $P'\in {\mathcal F}_{k+1}$ with $P'\subset P$. Denote $G_{P'}=P'\setminus \O(P)$.
Then, by \eqref{cz}, we have
$$|G_{P'}|\ge  \abs{P'} - \abs{P'\cap \Omega(P)} \geq \frac12|P'|.$$
Therefore, $|G_{P'}\setminus \O(P')|>0$ and hence, fixing
$$y\in G_{P'}\setminus \O(P')\subset P' \setminus\hab{\O_1(P) \cup \O_1(P')}$$
and applying \eqref{cond}, we obtain for a.e. $x\in P'$
\begin{eqnarray*}
|f_{P',P}(x)|&\le& \inf_{x' \in P'}m^{\#}_Pf(x')+|f_{P',P}(y)|\\
&\le& \big(m_P^{\#}f\big)^*\hab{|P|\tfrac{1-\eta}{2^{n+2}}}+|f_{P}(y)|+|f_{P'}(y)|\\
&\le& \ga_P+(f_{P'}\chi_{P'})^*\hab{|P'|\tfrac{1-\eta}{2^{n+2}}}.
\end{eqnarray*}

Combining this estimate with (\ref{Ep}) and Proposition \ref{tsd} yields for a.e. $x\in Q$ that
\begin{align*}
|f_Q(x)|^r&\le C_r^r\sum_{k=0}^{\infty}\sum_{P\in {\mathcal F}_k}\Big((f_P\chi_P)^*\hab{|P|\tfrac{1-\eta}{2^{n+2}}}^r\chi_{E_P}(x)\\
&\hspace{1cm}+\sum_{P'\in {\mathcal F}_{k+1}:P'\subset P}\big(\ga_P+(f_{P'}\chi_{P'})^*\hab{|P'|\tfrac{1-\eta}{2^{n+2}}}\big)^r\chi_{P'}(x)\Big)\\
&\le 2(2C_r)^r\sum_{k=0}^{\infty}\sum_{P\in {\mathcal F}_k}\ga_P^r\chi_P+(2C_r)^r\sum_{k=0}^{\infty}\sum_{P\in {\mathcal F}_k}\sum_{P'\in {\mathcal F}_{k+1}:P'\subset P}\ga_{P'}^r\chi_{P'}\\
&\le 3(2C_r)^r\sum_{P\in {\mathcal F}}\ga_P^r\chi_P,
\end{align*}
which completes the proof.
\end{proof}

\begin{remark}\label{localcase} It is easily seen from the proof that Theorem \ref{ofsdp} can be stated in a (formally stronger) local version. Namely, fix a cube $Q_0\in {\mathcal Q}$,
and assume that the $\ell^r$-condition and (\ref{cond}) hold for a family of functions $\{f_Q,f_{P,Q}\}$, where $P,Q\in {\mathcal D}(Q_0)$ and $P\subseteq Q$. Then for
any $\eta \in (0,1)$ there exists an $\eta$-sparse family ${\mathcal F}\subset \mc{D}(Q_0)$ such that for a.e. $x\in Q_0$,
$$|f_{Q_0}(x)|\lesssim C_r\Big(\sum_{P\in {\mathcal F}}\ga_P^r\chi_P(x)\Big)^{1/r}$$
with the same $\ga_P$ as in Theorem \ref{ofsdp}.
In particular, the family of all cubes ${\mathcal Q}$ in Theorem \ref{ofsdp} can be replaced by an arbitrary subfamily, for example, by a dyadic lattice ${\mathscr D}$.
\end{remark}

\subsection{A bilinear form sparse domination principle}
In certain applications the $m^{\#}_Qf$-term in the definition of $\gamma_Q$ in Theorem \ref{ofsdp} is too large to be efficiently estimated. We will therefore also prove a variant of Theorem \ref{ofsdp}, dominating $|f_Q|$ dually by a sparse form involving the smaller term $m^{\#}_{Q,q}f$ for $q \in (0,\infty)$.

\begin{theorem}\label{bf} Let $r \in (0,\infty)$ and let $\{f_Q, f_{P,Q}\}_{Q\in {\mathcal Q}, P\in {\mathcal D}(Q)}$ satisfy the $\ell^r$-condition and (\ref{cond}) and let $q \in (r,\infty)$.
For any $Q \in \mc{Q}$ and $\eta \in (0,1)$
there exists an $\eta$-sparse family ${\mathcal F}\subset {\mathcal D}(Q)$ such that for every measurable $g \colon \R^n \to \R_+$,
$$\int_Q|f_Q|^rg\lesssim C_r\sum_{P\in {\mathcal F}}\a_P^r\langle g\rangle_{(q/r)',P} |P|,$$
where
$$\a_P=(f_{P}\chi_{P})^*\hab{\abs{P}\tfrac{1-\eta}{2^{n+2}}}+(m^{\#}_{P,q}f)^*\hab{\abs{P}\tfrac{1-\eta}{2^{n+2}}}.$$
\end{theorem}

\begin{proof} The proof is similar to the proof of Theorem \ref{ofsdp} and hence some details are omitted.
Construct the $\eta$-sparse family ${\mathcal F}$ exactly as in the proof of Theorem~\ref{ofsdp}, only replacing $m_{P}^{\#}f$ by $m_{P,q}^{\#}f$ in the definition of $\O_2(P)$.

By Proposition \ref{tsd} we have
\begin{equation}\label{spl}
\begin{split}
\int_Q|f_Q|^rg\le C_r^r\sum_{k=0}^{\infty}\sum_{P\in {\mathcal F}_k} & \Big(\int_{E_P}|f_P|^rg+\sum_{\substack{P'\in {\mathcal F}_{k+1}:\\P'\subset P}}\int_{P'}|f_{P',P}|^rg\Big),
\end{split}
\end{equation}
and by (\ref{Ep})
\begin{equation}\label{fp}
\int_{E_P}|f_P|^rg\le \a_P^r\int_{P}g.
\end{equation}

It remains to estimate the second term on the right-hand side of (\ref{spl}). As in the proof of Theorem \ref{ofsdp},
we introduce the set $G_{P'}=P'\setminus \O(P)$ and observe that
\begin{equation}\label{mest}
|G_{P'}\setminus\O(P')|\ge \Big(\frac{1}{2}-\frac{1}{2^{n+3}}\Big)|P'|\ge \frac{1}{4}|P'|.
\end{equation}
For a.e. $y\in G_{P'}\setminus \O(P')\subset P' \setminus\hab{\O_1(P) \cup \O_1(P')}$ we have
\begin{align*}
\int_{P'}|f_{P',P}|^rg&\le 2^r\int_{P'}|f_{P',P}(x)-f_{P',P}(y)|^rg(x)\dd x+4^r(\a_{P'}^r+\a_{P}^r)\int_{P'}g.
\end{align*}
Integrating over $G_{P'}\setminus\O(P')$ and using (\ref{mest}), we obtain
\begin{align*}
\int_{P'}|f_{P',P}|^rg&\le 4\cdot2^r\frac{1}{|P'|}\int_{P'}\int_{P'}|f_{P',P}(x)-f_{P',P}(y)|^rg(x)\dd x\dd y\\
&\hspace{1cm}+4^r(\a_{P'}^r+\a_{P}^r)\int_{P'}g.
\end{align*}
By H\"older's inequality,
\begin{align*}
\frac{1}{|P'|}\int_{P'}\int_{P'}|f_{P',P}(x)&-f_{P',P}(y)|^rg(x)dxdy\\
&\le\osc_q(f_{P',P};P')^{r}\cdot \langle g\rangle_{(q/r)',P'}|P'|\\
&\le \inf_{x' \in P'}(m_{P,q}^{\#}f)(x')^r\langle g\rangle_{(q/r)',P'}|P'|\\&\le \a_{P}^r\langle g\rangle_{(q/r)',P'}|P'|,
\end{align*}
which, along with the previous estimate, implies
\begin{equation}\label{eq:fP'P}
\int_{P'}|f_{P',P}|^rg\le 5\cdot 4^r\Big(\a_{P}^r\langle g\rangle_{(q/r)',P'}|P'| + (\a_{P'}^r+\a_{P}^r)\int_{P'}g\Big).
\end{equation}

Now note that, by H\"older's inequality, we have
\begin{align*}
  \sum_{\substack{P'\in {\mathcal F}_{k+1}:\\P'\subset P}} \a_{P}^r\langle g\rangle_{(q/r)',P'}|P'| &\leq \a_{P}^r \has{\sum_{\substack{P'\in {\mathcal F}_{k+1}:\\P'\subset P}}  \int_{P'} g^{(q/r)'}}^{\frac{1}{(q/r)'}} \cdot \has{\sum_{\substack{P'\in {\mathcal F}_{k+1}:\\P'\subset P}}|P'|}^{\frac{r}{q}}\\
  &\leq \alpha_P^r \ip{g}_{(q/r)',P} \abs{P}
\end{align*}
Combining this estimate with \eqref{spl}, \eqref{fp} and \eqref{eq:fP'P}, we obtain
\begin{align*}
\int_Q|f_Q|^rg&\le 6\cdot 4^r\cdot C_r\Big(\sum_{k=0}^{\infty}\sum_{P\in {\mathcal F}_k}\Big(\a_P^r\int_Pg+\sum_{P'\in {\mathcal F}_{k+1}:P'\subset P}\a_{P'}^r\int_{P'}g\Big)\\&\hspace{1cm}+\sum_{k=0}^{\infty}\sum_{P\in {\mathcal F}_k}\a_{P}^r\langle g\rangle_{(q/r)',P}|P|\Big)\\
&\le 18\cdot  4^r \cdot C_r\sum_{P\in {\mathcal F}}\a_{P}^r\langle g\rangle_{(q/r)',P}|P|,
\end{align*}
which completes the proof.
\end{proof}

\subsection{Sparse domination in spaces of homogeneous type}
 A space of homogeneous type $(S,d,\mu)$, originally introduced by Coifman and Weiss in \cite{CW71}, is a set $S$ equipped with a quasi-metric $d$ and a doubling Borel measure $\mu$. That is,
 a metric $d$  which instead of the triangle inequality satisfies
\begin{equation*}
  d(s,t) \leq c_d\, \hab{d(s,u)+d(u,t)}, \qquad s,t,u\in S
\end{equation*}
for some $c_d\geq 1$, and a Borel measure $\mu$ that satisfies the doubling ball property
\begin{equation*}
  \mu\hab{B(s,2\rho)} \leq c_\mu \,\mu\hab{B(s,\rho)}, \qquad s \in S,\quad \rho>0
\end{equation*}
for some $c_\mu\geq 1$.

It was shown by Anderson and Vagharshakyan \cite{AV14} that the sparse domination principle based on the median oscillation estimate \eqref{lmo} could be generalised from the Euclidean space $\R^n$ equipped with the Lebesgue measure to a space of homogeneous type. Later, in \cite{Lo19b}, it was shown by the  second author that the sparse domination principle for operators in \cite{Le16, LO19} also generalises directly to spaces of homogeneous type. Doing similar adaptations as in \cite{Lo19b}, Theorems \ref{ofsdp} and \ref{bf} also generalise to this setting.

\section{Previous known results from our sparse domination principles}\label{sec:oldresults}
In this section we will show that Theorems \ref{ofsdp} and \ref{bf} imply a number of the previously known results.

\subsection{The local mean oscillation estimate} Let us start by showing that Theorem \ref{ofsdp} implies (\ref{lmo}) or, more generally, its vector-valued variant by H\"anninen and Hyt\"onen \cite{HH14}.

Let $X$ be a Banach space and $f \colon \R^n\to X$ be a strongly measurable function. Given $0<\la<1$, the local mean oscillation of $f$ on $Q \in \mc{Q}$ is defined by
$$\o_{\la}(f;Q)=\inf_{c \in X} \big(\nrm{f-c}_X\chi_Q\big)^*\big(\la|Q|\big).$$
Moreover, for $0<\lambda<\frac12$, we define the quasi-optimal center of oscillation $c_{\lambda}(f;Q)$ as any vector $c \in X$ such that
\begin{equation*}
  \big(\nrm{f-c}_X\chi_Q\big)^*\big(\la|Q|\big) \leq 2\, \o_{\la}(f;Q),
\end{equation*}
see \cite[Section 4]{HH14} for an introduction.
We will use the following property of this object:
for $0<\lambda \leq \kappa <\frac12$ we have for any quasi-optimal center of oscillation $c_\kappa(f;Q)$ that
\begin{equation}\label{eq:quasiincreasing}
  \big(\nrm{f-c_\kappa(f;Q)}_X\chi_Q\big)^*\big(\la|Q|\big) \leq 4\, \o_{\la}(f;Q),
\end{equation}
see \cite[Lemma 4.10]{HH14}.

\begin{theorem}[\cite{HH14}]\label{locmeanosc} Let $X$ be a Banach space and let $f\colon \R^n \to X$ be strongly measurable. For every cube $Q\in {\mathcal Q}$ and $\eta \in (0,1)$,
there exists an $\eta$-sparse family ${\mathcal F}\subset {\mathcal D}(Q)$ such that for any $c_{1/4}(f;Q)$ and for a.e. $x\in Q$,
$$\nrm{f(x)-c_{1/4}(f;Q)}_X\lesssim \sum_{P\in {\mathcal F}}\o_{\frac{1-\eta}{2^{n+2}}}(f;P)\chi_P(x).$$
\end{theorem}

\begin{proof} For $Q\in {\mathcal Q}$ and $P\in {\mathcal D}(Q)$ define
\begin{align*}
  f_Q&:=\nrm{f-c_{1/4}(f;Q)}_X,\\
  f_{P,Q}&:=\nrm{c_{1/4}(f;P) - c_{1/4}(f;Q)},
\end{align*}
where for any cube $Q \in \mc{Q}$ we fix a quasi-optimal center of oscillation $c_{1/4}(f;Q)$.

The family  $\{f_Q, f_{P,Q}\}_{Q \in \mc{Q},P \in \mc{D}(Q)}$ satisfies the $\ell^1$-condition with ${C_1=1}$ by the triangle inequality, and (\ref{cond}) holds as well.
Therefore, by Theorem \ref{ofsdp}, there exists an $\eta$-sparse family ${\mathcal F}\subset {\mathcal D}(Q)$ such that for a.e. $x\in Q$,
$$\nrm{f(x)-c_{1/4}(f;Q)}_X \lesssim \sum_{P\in {\mathcal F}} \gamma_P \chi_P(x).$$

The function $f_{P',P}$ is a constant for every $P'\in {\mathcal D}(P)$, and therefore $m^{\#}_Pf=0$. Hence, by \eqref{eq:quasiincreasing},
$$\ga_P=\hab{\nrm{f-c_{1/4}(f;P)}_X\chi_{P}}^*\hab{|P|\tfrac{1-\eta}{2^{n+2}}}\le 4\o_{\frac{1-\eta}{2^{n+2}}}(f;P),$$
which completes the proof.
\end{proof}

\subsection{The \texorpdfstring{$\ell^r$}{lr}-sparse domination principle for operators}\label{subsection:ellr} Next we show that Theorem \ref{ofsdp} implies the main result of the second author in \cite[Theorem 3.2]{Lo19b} on pointwise $\ell^r$-sparse domination for a bounded operator $T\colon L^p \to L^{p,\infty}$. Moreover, due to our flexible setup, we also obtain the pointwise $\ell^r$-sparse domination result for bounded operators $T\colon L^p \to L^{q,\infty}$ in \cite[Theorem 3.4]{Lo19b} without any additional effort.

Let us first introduce some notation. Let ${\mathscr D}$ be a dyadic lattice, let $X,Y$ be Banach spaces, $p,q \in (0,\infty)$ and $\a\ge 1$. For a bounded operator
$$T\colon L^p(\R^n;X) \to L^{q,\infty}(\R^n;Y)$$ we say that a family of operators $\{T_Q\}_{Q\in {\mathscr D}}$ from $L^p(\R^n;X)$ to $L^{q,\infty}(Q;Y)$
is an \emph{$\a$-localisation family} of $T$ if for all $Q\in {\mathscr D}$ and $f\in L^p(\R^n;X)$ we have
\begin{align*}
  T_Q(f\chi_{\a Q})(x)&=T_Qf(x), && x\in Q,\\
  \|T_Q(f\chi_{\a Q})(x)\|_{Y}&\le \|T(f\chi_{\a Q})(x)\|_{Y},&&x\in Q.
\end{align*}
The canonical example is, of course, $T_Qf:=T_Q(f\chi_{\a Q})$ for all $Q \in \ms{D}$.

Using an $\alpha$-localisation family of $T$, we can define operator variants of the $\ell^r$-condition and the sharp maximal function $m_Q^{\#}f$. Indeed, set $T_{P,Q}:= T_Q-T_P$ and let $r \in (0,\infty)$.
We say that $T$ satisfies a \emph{localised $\ell^r$-estimate} if
for every $Q\in {\mathcal Q}$ and every $P_1,\dots, P_m\in {\mathcal D}(Q)$ with $P_m\subset\dots\subset P_1$ and $f \in L^p(\R^n;X)$, we have for a.e. $x \in P_m$
$$\|T_{P_1}f(x)\|_{Y}\le C_r\has{\sum_{k=1}^{m-1}\|T_{P_{k+1},P_{k}}f(x)\|_{Y}^r+\|T_{P_m}f(x)\|_{Y}^r}^{1/r}.$$

Observe that if $r \in (0,1]$ and $T_Qf = T(f \chi_{\alpha Q})$ for some $\alpha\geq 1$, then this estimate holds trivially for every (sub)linear operator $T$ with $C_r=1$.
If $T$ satisfies a localised $\ell^r$-estimate, then, setting
\begin{align}
  \label{eq:fQoperator} f_Q(x)&:=\|T_Qf(x)\|_Y&&  x \in \R^n\\
  \label{eq:fPQoperator}f_{P,Q}(x)&:=\|T_{P,Q}f(x)\|_{Y}, &&  x \in \R^n
\end{align}
for $Q \in \ms{D}$ and $P \in \mc{D}(Q)$, we obtain that the $\ell^r$-condition holds.

Next, the operator analogue of the sharp maximal function $m_Q^{\#}f$ for a cube $Q \in \mc{Q}$ is given by
$${\mathcal M}_{T,Q}^{\#}f(x)=\sup_{P\in {\mathcal D}(Q):x \in P}\esssup_{x',x''\in P}\,\nrmb{T_{P,Q}f(x')-T_{P,Q}f(x'')}_Y,\qquad x \in \R^n.$$
For $f_{P,Q}$ as in \eqref{eq:fPQoperator} we have
$$|f_{P,Q}(x')-f_{P,Q}(x'')|\le \|T_{P,Q}f(x')-T_{P,Q}f(x'')\|_Y,\qquad x',x''\in P$$
and therefore
\begin{equation}\label{twosharpfun}
m_Q^{\#}f(x)\le {\mathcal M}_{T,Q}^{\#}f(x), \qquad x \in \R^n.
\end{equation}

We are now ready to prove the announced result from \cite{Lo19b}, which in the diagonal case $p=q$ generalises the main result of \cite{LO19}.

\begin{theorem}[\cite{Lo19b}]\label{ellrsdp}
Let $\ms{D}$ be a dyadic lattice, let $X$ and $Y$ be Banach space, take $p,q,r \in (0,\infty)$ and let $\a\ge 1$. Suppose that
\begin{itemize}
\item $T$ is a bounded operator from $L^{p}(\R^n;X)$ to $L^{q,\infty}(\R^n;Y)$ with $\a$-localisation family $\{T_Q\}_{Q\in {\mathscr D}}$.
\item ${\mathcal M}_{T,Q}^{\#}$ is bounded from $L^{p}(\R^n;X)$ to $L^{q,\infty}(\R^n)$ uniformly in ${Q\in{\mathscr D}}$.
\item $T$ satisfies a localised $\ell^r$-estimate.
\end{itemize}
Then for any $f\in L^{p}(\R^n;X)$ and $Q\in {\mathscr D}$, there exists a $\frac{1}{2}$-sparse family ${\mathcal F}\subset {\mathcal D}(Q)$ such that for a.e. $x \in Q$,
$$\|T_Qf(x)\|_{Y}\lesssim  C_T \, C_r\, \Big(\sum_{P\in {\mathcal F}}\abs{\alpha P}^{\frac{r}{p}-\frac{r}{q}}\ip{\|f\|_X}_{p, \a P}^r\chi_P(x)\Big)^{1/r},$$
with $C_T:= \nrm{T}_{L^p\to L^{q,\infty}} + \sup_{Q \in \ms{D}} \nrm{\mc{M}_{T,Q}}_{L^p\to L^{q,\infty}}.$
\end{theorem}

\begin{proof}
The essence of the proof is already contained in the discussion preceding the theorem. Indeed, let $\cbrace{f_Q,f_{P,Q}}_{Q \in \ms{D}, P \in \mc{D}(Q)}$ be as in \eqref{eq:fQoperator} and \eqref{eq:fPQoperator}, which satisfies the $\ell^r$-condition, and also (\ref{cond}) holds.
Hence we are in position to apply Theorem \ref{ofsdp} with $\eta=\frac{1}{2}$ (see also Remark~\ref{localcase}). It remains to estimate $\ga_P$ provided by this theorem.

By the assumption on $T$ we have
$$(f_P\chi_P)^*(|P|/2^{n+3})\lesssim \nrm{T}_{L^p\to L^{q,\infty}} \cdot \abs{\alpha P}^{\frac{1}{p}-\frac{1}{q}}  \langle\|f\|_X\rangle_{p, \a P}.$$
Moreover, by the assumption on ${\mathcal M}_{T,P}^{\#}$ and (\ref{twosharpfun}), we have
$$(m^{\#}_P\chi_P)^*(|P|/2^{n+3})\lesssim \nrm{\mc{M}_{T,Q}}_{L^p\to L^{q,\infty}} \cdot \abs{\alpha P}^{\frac{1}{p}-\frac{1}{q}}\langle\|f\|_X\rangle_{p, \a P},$$
which completes the proof.
\end{proof}

\begin{remark}
  In \cite[Theorem 3.5]{Lo19b} a sparse \emph{form} domination principle was shown for operators. Analogous to the proof of Theorem \ref{ellrsdp}, one can deduce \cite[Theorem 3.5]{Lo19b} from our sparse form domination principle in Theorem \ref{bf}.
\end{remark}

\begin{remark}\label{knownapp}
Both papers \cite{LO19} and \cite{Lo19b} contain a list of known sparse domination results for operators that fit our setting.
For reader's convenience, we include a unified and extended list below.
\begin{itemize}
\item Calder\'on--Zygmund operators \cite{LO19} with operator-valued kernel \cite{Lo19b}.
\item Maximally modulated Calder\'on--Zygmund operators \cite{Be17}.
\item Variational truncations of Calder\'on--Zygmund operators \cite{MTX15, Zo20}.
\item Multilinear singular integral operators with $L^r$-H\"ormander condition \cite{Li18}.
\item Fractional integral operators with Horm\"ander kernel \cite{IRV18}.
\item A class of pseudo-differential operators \cite{BC20}.
\item The Rademacher \cite{Lo19b} and the lattice Hardy--Littlewood \cite{HL17} maximal operators.
\item The intrinsic Littlewood--Paley square function \cite{Lo21b}.
\item Nonintegral operators falling outside the scope of Calder\'on--Zyg\-mund theory \cite{BFP16}
and the associated square functions \cite{BBR20}.
\item Rough homogeneous singular integrals \cite{Le19}.
\item The Marcinkiewicz integral with rough kernel \cite{TH}.
\end{itemize}

Let us note that the boundedness of the corresponding sharp maximal operator is not explicitly contained in all of the above cited works and, in particular, sparse domination is obtained in a self-contained way in many of these citations. However, the presented arguments often imply the boundedness of the corresponding sharp maximal operator in our setting. For more details we refer to \cite[Section 5]{LO19} and \cite[Section 9]{Lo19b}.

The three last items from the list fit the setting of bilinear form sparse domination expressed in Theorem \ref{bf}.
\end{remark}

\section{Generalised Poincar\'e--Sobolev inequalities}\label{sec:poincare}
As a first new application of our operator-free sparse domination principle, we will study generalised Poincar\'e--Sobolev inequalities as in~\cite{CMPR21,PR19}. In particular, we will extend and improve
\cite[Theorem 1.5 and 1.24]{PR19} by P\'erez and Rela.

Let us introduce some notation. Let $p,s \in [1,\infty)$. For a functional $a \colon \mc{Q} \to \R_+$ and a weight $w$ we say that $a$ satisfies the $SD_p^s(w)$-condition, and write $a \in SD_p^s(w)$, if for any cube $Q \in \mc{Q}$ and any family of pairwise disjoint $\cbrace{Q_j} \subset \mc{D}(Q)$ we have
\begin{equation*}
  \has{\frac{1}{w(Q)}\sum_{j} a(Q_j)^p w(Q_j)}^{1/p} \leq C\has{\frac{\sum_{j}\abs{Q_j}}{\abs{Q}}}^{{1/s}} a(Q).
\end{equation*}
The least admissible constant $C\geq 1$ is denoted by $\nrm{a}_{SD_p^s(w)}$.
We note that the $SD_p^s(w)$-condition can be thought of as an $s$-smallness preserving condition and for examples of functionals $a \in SD_p^s(w)$ we refer to~\cite{PR19}.

Fix a cube $Q \in \mc{Q}$ and $f \in L^1_{\loc}(\R^n)$. For $m \in \N \cup \cbrace{0}$, we denote by $P_Qf$ the projection of $f$ onto the space of polynomials of degree at most $m$ in $n$ variables on $Q$. We refer to \cite[Section 8]{PR19} for a proper introduction of this projection. Here we just note the following two properties that we will use of $P_Qf$:
\begin{itemize}
  \item There is a $C_m>0$ such that
  \begin{equation}\label{eq:P_Qaverage}
    \nrm{P_Qf}_{L^\infty(Q)} \leq C_m \frac{1}{\abs{Q}} \int_Q\abs{f}.
  \end{equation}
  \item For any polynomial $\pi$ of degree at most $m$ in $n$ variables we have $P_Q(\pi) = \pi$  on $Q$.
\end{itemize}
Furthermore we note that when $m=0$, we have $P_Qf =\frac{1}{|Q|}\int_Qf$.

In this language, the main result of P\'erez and Rela reads as follows:

\begin{theorem}[\cite{PR19}]\label{PerezRela} Let $p,s\in [1,\infty)$ and let $w\in A_{\infty}$. Assume that $a\in SD_p^s(w)$.
Let $f\in L^1_{\loc}({\mathbb R}^n)$ be such that for all $Q \in \mc{Q}$,
\begin{equation*}
    \frac{1}{\abs{Q}}\int_Q \abs{{f - P_Qf}} \leq a(Q).
  \end{equation*}
Then,  there is a constant $C_{n,m}>0$ such that for any $Q \in \mc{Q}$

    \begin{equation}\label{sdep}
    \has{\frac{1}{w\ha{Q}}\int_Q \abs{{f - P_Qf}}^pw}^{\frac1p} \leq C_{n,m}\,  s \, \nrm{a}^s_{SD_p^s(w)} \, a(Q).
  \end{equation}
\end{theorem}

Note that when $m \geq 1$, the result of P\'erez and Rela has an additional factor  $2^{\frac{s+1}{p'}}$ in the conclusion, but it was observed in \cite[Theorem 2.1]{CMPR21} that this factor can be omitted.

It was asked in \cite[Remark 1.6]{PR19} whether the $A_{\infty}$ assumption in Theorem \ref{PerezRela} can be removed. A partial result in this direction was provided by Mart\'inez-Perales \cite{Ma19}.

In order to state our main result, we will replace $L^p(w)$-averages by arbitrary Banach function
norms (see e.g. \cite{BS88,Za67}). First we define a more general smallness preserving condition with respect to a Banach function norm. Note that the following condition with
\begin{align}\label{eq:perezrelasetting}
\begin{aligned}
    \nrm{f}_{X_Q} &:= \has{\frac{1}{w(Q)}\int_Q \abs{f}^pw}^{1/p}, \qquad && Q \in \mc{Q},\\
  \varphi(t) &:= \nrm{a}_{SD_p^s} \cdot t^{\frac{1}{s}}, && t \in [0,1],
\end{aligned}
\end{align}
 coincides with the definition of the $SD_p^s(w)$-condition.

\begin{definition}\label{genacond} For $Q \in \mc{Q}$ let $\nrm{\,\cdot\,}_{X_Q}$ be a Banach function norm and let $\varphi\colon [0,1] \to \R_+$ be increasing. For a functional $a \colon \mc{Q} \to \R_+$ we say that $a$ satisfies the $\varphi$-smallness preserving condition if for any $Q \in \mc{Q}$ and any family of pairwise disjoint $\cbrace{Q_j} \subset \mc{D}(Q)$ we have
\begin{equation}\label{genacond1}
\nrmb{\sum_{j}a(Q_j)\chi_{Q_j}}_{X_Q}\leq \varphi\has{\frac{\sum_{j}\abs{Q_j}}{\abs{Q}}}\cdot a(Q).
\end{equation}
\end{definition}

We are now ready to state the main result of this section.

\begin{theorem}\label{genbanachx} Fix $f \in L^1_{\loc}(\R^n)$.
 For $Q \in \mc{Q}$ let $\nrm{\,\cdot\,}_{X_Q}$ be a Banach function norm and let  $a \colon \mc{Q} \to \R_+$ satisfy both the $\varphi$-smallness preserving condition and for all $Q \in \mc{Q}$
  \begin{equation*}
    \frac{1}{\abs{Q}}\int_Q \abs{{f - P_Qf}} \leq a(Q).
  \end{equation*}
Then there is a $C_{n,m}>0$ such that for all $Q \in \mc{Q}$
\begin{equation}\label{lins}
\|(f-P_Qf)\chi_Q\|_{X_Q}\le C_{n,m} \,a(Q) \cdot \has{\int_0^1 \varphi(t) \tfrac{\ddn t}{t} +\varphi(1)}.
\end{equation}
\end{theorem}

Taking $X_Q$ and $\varphi$ as in \eqref{eq:perezrelasetting}, we have
\begin{equation*}
  \int_0^1 \varphi(t) \tfrac{\ddn t}{t} +\varphi(1)= (s+1) \, \nrm{a}_{SD_p^s}.
\end{equation*}
Thus we obtain an extension of Theorem~\ref{PerezRela} to arbitrary weights, which provides an affirmative answer to the question posed in \cite[Remark 1.6]{PR19}. Moreover, we have a quantitative improvement over Theorem \ref{PerezRela}, even in the case $m=0$ and $s>1$,
since (\ref{lins}) holds with linear dependence on $\nrm{a}_{SD_p^s(w)}$, whereas one has  $\nrm{a}^s_{SD_p^s(w)}$ in (\ref{sdep}).

The key ingredient in our proof of Theorem \ref{genbanachx} is the following sparse domination result in the spirit of Theorem \ref{locmeanosc}.

\begin{prop}\label{sparsepolyn}
Let $f \in L^1_{\loc}(\R^n)$. For any $Q \in \mc{Q}$ and $\eta \in (0,1)$
there exists an $\eta$-sparse family ${\mathcal F}\subset \mc{D}(Q)$ such that
$$
|f-P_Qf| \chi_Q\le C_{n,m}\frac{1}{1-\eta}\sum_{R\in {\mathcal F}}\left(\frac{1}{\abs{R}}\int_R\abs{f - P_Rf}\right)\chi_R.
$$
\end{prop}

\begin{proof}
For $Q \in \mc{Q}$ and $R \in \mc{D}(Q)$ define
\begin{align*}
  f_Q &:= {{f - P_Qf}}\\
  f_{R,Q} &:={{P_Rf - P_Qf}}.
\end{align*}
The family  $\{f_Q,f_{R,Q}\}_{Q\in {\mathcal Q}, R\in {\mathcal D}(Q)}$ trivially satisfies the $\ell^1$-condition with $C_r=1$, and (\ref{cond}) holds.

For any $R' \in \mc{D}(R)$ we have by \eqref{eq:P_Qaverage}
\begin{equation*}
  \nrm{f_{R',R}}_{L^\infty(R')} =  \nrm{P_{R'}(f-P_{R}f)}_{L^\infty(R')} \leq C_m \, \frac{1}{\abs{R'}} \int_{R'} \abs{f - P_{R}f},
\end{equation*}
which implies
\begin{equation*}
  m^{\#}_Rf(x) \leq 2 C_mM \hab{(f-P_Rf)\chi_R}(x), \qquad x \in \R^n.
\end{equation*}
Therefore, by Chebyshev's inequality and the weak $L^1$-boundedness of~$M$, we have for any $\eta \in (0,1)$,
\begin{equation*}
  (f_{R}\chi_{R})^*\hab{|R|\tfrac{1-\eta}{2^{n+2}}}+(m^{\#}_Rf)^*\hab{|R|\tfrac{1-\eta}{2^{n+2}}} \leq C_{n,m} \, \frac{1}{1-\eta}\cdot{\frac{1}{\abs{R}} \int_R \abs{f - P_Rf}},
\end{equation*}
which, by Theorem \ref{ofsdp}, completes the proof.
\end{proof}

\begin{proof}[Proof of Theorem \ref{genbanachx}]
Fix a cube $Q\in {\mathcal Q}$. By the main hypothesis of Theorem \ref{genbanachx} combined with Proposition \ref{sparsepolyn},
there exists a $\frac12$-sparse family ${\mathcal F}\subset \mc{D}(Q)$ such that
\begin{equation}\label{fPQ}
|f-P_Qf|\le C_{n,m} \sum_{R\in {\mathcal F}}a(R)\chi_R.
\end{equation}

Write $\mc{F} = \bigcup_{k=0}^\infty \mc{F}_{k}$, where $\mc{F}_k$ is as in the definition of a contracting family of dyadic cubes. Since $\mc{F} $ is $\frac12$-sparse, we have for any $k \in \N\cup \cbrace{0}$
  \begin{equation*}
    \sum_{R \in \mc{F}_{k}} \abs{R} \leq \frac1{2^k} \abs{Q},
  \end{equation*}
which, along with the $\varphi$-smallness preserving condition, implies
\begin{align*}
\nrmb{\sum_{R \in \mc{F}}a(R)\chi_R}_{X_Q}&\le \sum_{k=0}^\infty \nrmb{\sum_{R \in \mc{F}_k}a(R)\chi_R}_{X_Q} \le a(Q) \,\sum_{k=0}^{\infty} \varphi\ha{2^{-k}}.
\end{align*}
Combined with \eqref{fPQ}, this implies
$$\|(f-P_Qf)\chi_Q\|_{X_Q}\le C_{n,m} \, a(Q) \,\sum_{k=0}^{\infty} \varphi\ha{2^{-k}}.$$
The result now follows by noting $\sum_{k=1}^{\infty} \varphi\ha{2^{-k}} \leq \int_0^1 \varphi(t) \tfrac{\ddn t}{t}$.
\end{proof}

\begin{remark}\label{quasibanach}
Theorem \ref{genbanachx} remains true for quasi-Banach function norms. In this case one has to replace $\int_0^1 \varphi(t) \tfrac{\ddn t}{t}$ by $\hab{\int_0^1 \varphi(t)^r \tfrac{\ddn t}{t}}^{1/r}$, where $r\in (0,1)$ is the exponent in the Aoki--Rolewicz theorem (see \cite{KPR84}).
\end{remark}

\begin{remark}\label{MQ} One can replace $\|(f-P_Qf)\chi_Q\|_{X_Q}$ in the left-hand side of the conclusion of Theorem \ref{genbanachx} by $\|M_Q(f-P_Qf)\|_{X_Q}$, where $M_Q$ is the local maximal operator given by
\begin{equation*}
  M_Qf := \sup_{P \in \mc{D}(Q)} \ip{f}_{1,P} \chi_P.
\end{equation*}
Indeed, one can make a similar change in Proposition \ref{sparsepolyn} by using
\begin{align*}
f_Q &:= M_Q(f-P_Qf),\\
f_{R,Q} &:= f_R-f_Q
\end{align*}
 in the proof. The usage of Chebyshev's inequality is in this case replaced by the weak $L^1$-boundedness of $M_Q$.
\end{remark}

Using Remark \ref{MQ}, one can recover e.g. the first main result of \cite{CP21}. For a weight $w$, a cube $Q \in \mc{Q}$ and $r>0$ denote
$$
w_r(Q) := \abs{Q}^{1/r'} \has{\int_Qw^r}^{1/r}.
$$
Furthermore, for $f \in L^1_{\loc}(\R^n)$ define the polynomial sharp maximal function as
$$
M^{\sharp}_mf(x):= \sup_{Q \ni x} \frac{1}{\abs{Q}} \int_Q\abs{f-P_Qf}.
$$

\begin{cor}[\cite{CP21}]
Let $f \in L^1_{\loc}(\R^n)$, let $w$ be a weight, take $p \in [1,\infty)$ and $r \in (1,\infty)$. For any cube $Q \in \mc{Q}$ we have
\begin{equation*}
\has{\frac{1}{w_r(Q)} \int_Q \has{\frac{M_Q(f-P_Qf)}{M_m^{\sharp}f}}^pw}^{1/p} \leq C_{n,m}\, pr'
\end{equation*}
\end{cor}

\begin{proof}
This follows directly from Theorem \ref{genbanachx} combined with Remark \ref{MQ} using the choices
\begin{align*}
a(Q) &= \frac{1}{\abs{Q}} \int_Q\abs{f-P_Qf},\\
\nrm{g}_{X_Q} &=  \has{\frac{1}{w_r(Q)} \int_Q \has{\frac{\abs{g}}{M_m^{\sharp}f}}^pw}^{1/p} \cdot a(Q).
\end{align*}
Indeed,  for any $Q \in \mc{Q}$ and any family of pairwise disjoint $\cbrace{Q_j} \subset \mc{D}(Q)$ we have by H\"older's inequality
\begin{align*}
\nrmb{\sum_{j}a(Q_j)\chi_{Q_j}}_{X_Q}&\leq \has{\frac{1}{w_r(Q)}  \sum_{j}\int_{Q_j} w }^{1/p}\cdot a(Q) \\&\leq
\has{\frac{\sum_{j}\abs{Q_j}}{\abs{Q}}}^{\frac{1}{pr'}}\cdot a(Q).
\end{align*}
so $a$ satisfies the $\varphi$-smallness preserving condition with $\varphi(t) = t^{\frac{1}{pr'}}$.
\end{proof}

\section{Tent spaces}\label{sec:tent}
As our second new application, we will use our sparse domination principle to prove the main relation between two basic operators in the theory of tent spaces.

Let ${\mathbb R}^{n+1}_+=\{(y,t): y\in {\mathbb R}^n, t>0\}$ and, given $\a>0$, let $\G_{\a}(x)$ denote the cone in ${\mathbb R}^{n+1}_+$
with vertex in $x\in {\mathbb R}^n$ of aperture $\a$, i.e.
$$\G_\alpha(x)=\{(y,t)\in {\mathbb R}^{n+1}_+:|x-y|<\a t\}.$$
Given a ball $B=B(x,r)$ in ${\mathbb R}^n$, denote the tent over $B$ by
$$\widehat B=\{(y,t)\in {\mathbb R}^{n+1}_+:|x-y|+t<r\}.$$
For a measurable function $f\colon {\mathbb R}^{n+1}_+ \to \R$ define
\begin{align*}
  A^{(\a)}(f)(x)&:=\has{\int_{\G_{\a}(x)}|f(y,t)|^2\frac{\ddn y\ddn t}{t^{n+1}}}^{1/2}, &&x \in \R^n,\\
  C(f)(x)&:=\sup_{B\ni x}\has{\frac{1}{|B|}\int_{\widehat B}|f(y,t)|^2\frac{\ddn y\ddn t}{t}}^{1/2}, &&x \in \R^n,
\end{align*}
where the supremum is taken over all balls $B\subset \R^n$ containing $x$.

In \cite{CMS85}, Coifman, Meyer and Stein defined the tent space $T^{p}_\alpha$ for $p\in (0,\infty)$ and $\alpha>0$ as the space of all measurable $f\colon {\mathbb R}^{n+1}_+ \to \R$ such that
$$\|f\|_{T^{p}_\alpha}:= \nrm{A^{(\alpha)}(f)}_{L^p(\R^n)}<\infty.$$
It was shown in \cite{CMS85} that $T^{p}_\alpha = T^{p}_\beta$ for $\alpha,\beta>0$ and thus it suffices to study $T^{p}:=T^{p}_1$. Furthermore, they deduced
\begin{align}
 \label{eq:TunderC} \nrm{f}_{T^p} &\lesssim \nrm{C(f)}_{L^p(\R^n)} && p \in (0,\infty),\\
 \label{eq:CunderT} \nrm{C(f)}_{L^p(\R^n)} &\lesssim \nrm{f}_{T^p}  && p \in (2,\infty).
\end{align}
To prove these inequalities, it is useful to define a truncated version of $A^{(\alpha)}$, i.e. for $h>0$ set
\begin{align*}
  A^{(\a)}_h(f)(x)&:=\has{\int_{0}^h\int_{\abs{x-y}<\alpha t}|f(y,t)|^2\frac{\ddn y\ddn t}{t^{n+1}}}^{1/2}, &&x \in \R^n,
\end{align*}
and note that, using Fubini's theorem, we can reformulate $C(f)(x)$ for $x \in \R^n$ as follows
\begin{align}
\notag C(f)(x) &\eqsim \sup_{x \ni B} \has{\frac{1}{\abs{B}} \int_{0}^{r(B)} \int_{B} \abs{f(y,t)}^2 \frac{\abs{B(y,\alpha t)}}{t^n} \frac{\ddn y \ddn t}{t}}\\
\label{eq:reformC} &\eqsim  \sup_{x \ni B} \has{\frac{1}{\abs{B}} \int_B \int_0^{r(B)} \int_{\abs{y-z} \leq \alpha t}\abs{f(y,t)}^2 \frac{\ddn y \ddn t}{t^{n+1}}\dd z}^{1/2}\\
\notag&= \sup_{x \ni B} \has{\frac{1}{\abs{B}} \int_B A_{r(B)}^{(\alpha)}(f)(z)^2 \dd z}^{1/2},
\end{align}
where $r(B)$ denotes the radius of the ball $B$ and the implicit constants depend on $\alpha>0$.

From \eqref{eq:reformC} it is clear that $C(f)^2 \lesssim M(A(f)^2)$, which directly implies \eqref{eq:CunderT} by the boundedness of the maximal operator.
We will give a ``sparse" proof of the converse in \eqref{eq:TunderC}.

\begin{theorem}\label{theorem:sparsetent} Take $\alpha>0$ and let $f \colon \R^{n+1}_+ \to \R$ be measurable.  For every cube $Q\in \mc{Q}$ there exists a $\frac{1}{2}$-sparse family ${\mathcal F}\subset {\mathcal D}(Q)$ such that for a.e. $x \in Q$,
$$A_{\ell_Q}^{(\alpha)}(f)(x)\lesssim \has{\sum_{P\in {\mathcal F}} \frac{1}{\abs{P}} \int_P A_{\ell_P}^{(4\alpha+\sqrt{n})}(f)^2\cdot \chi_P(x)}^{1/2}.$$
\end{theorem}

Combining Theorem \ref{theorem:sparsetent} with \eqref{eq:reformC}, we obtain for $\delta \in (0,2]$
and $g \in L^{p'}(\R^n)$,
\begin{align*}
\int_Q A_{\ell_Q}^{(1)}(f)^{\delta}g&\lesssim \sum_{P\in {\mathcal F}}\has{\frac{1}{|P|} \int_P A_{\ell_P}^{(4+\sqrt{n})}(f)^2}^{\d/2}\int_P g\\
&\lesssim \sum_{P\in {\mathcal F}}\int_{E_P}(C(f))^{\delta}\cdot Mg \\&\lesssim\|C(f)^{\delta}\|_{L^p(\R^n)}\|g\|_{L^{p'}(\R^n)}.
\end{align*}
By duality and the monotone convergence theorem, this yields \eqref{eq:TunderC}.

\begin{proof}[Proof of Theorem \ref{theorem:sparsetent}] Let $\Phi$ be a smooth function such that
$\chi_{B(0,1)}\le \Phi\le \chi_{B(0,2)}$ and for $Q \in \mc{Q}$ define
\begin{equation*}
  f_Q(x):= \has{\int_0^{\ell_Q} \int_{\R^n} \abs{f(y,t)}^2 \cdot \Phi\hab{\tfrac{x-y}{\alpha t}}^2
  \frac{\ddn y\ddn t}{t^{n+1}}}^{1/2}, \qquad x \in Q.
\end{equation*}
Observe that $A^{(\alpha)}_{\ell^Q}(f) \leq f_Q \leq A^{(2\alpha)}_{\ell^Q}(f)$. For $P \in \mc{D}(Q)$ set
\begin{equation*}
  f_{P,Q}(x):= \has{\int_{\ell_P}^{\ell_Q} \int_{\R^n} \abs{f(y,t)}^2 \cdot \Phi\hab{\tfrac{x-y}{\alpha t}}^2
  \frac{\ddn y\ddn t}{t^{n+1}}}^{1/2},  \qquad x \in P.
\end{equation*}
The family $\{f_Q,f_{P,Q}\}_{Q \in \mc{Q}, P \in \mc{D}(Q)}$ trivially satisfies the $\ell^2$-condition with $C_2=1$, and condition (\ref{cond}) holds as well.
Therefore, by Theorem \ref{ofsdp}, there exists a $\frac{1}{2}$-sparse family ${\mathcal F}\subset {\mathcal D}(Q)$ such that for a.e. $x \in Q$,
\begin{equation}\label{eq:sparsedomtent}
  A^{(\alpha)}_{\ell_Q}(f)(x)\le f_Q(x)\lesssim \has{\sum_{P\in {\mathcal F}}\ga_P^2\chi_P(x)}^{1/2},
\end{equation}
where
$$\ga_P=(f_{P}\chi_{P})^*(|P|/2^{n+3})+(m^{\#}_Pf)^*(|P|/2^{n+3}).$$

We start by analysing $m_P^{\#}f$. Fix $P \in \mc{F}$ and $x \in P$. Let $R\in {\mathcal D}(P)$ be such that $x\in \R$ and take $N \in \N$ such that $2^N\ell_R=\ell_P$. We have for $\xi,\eta \in R$
\begin{align*}
|f_{R,P}(\xi)&-f_{R,P}(\eta)|\\& \leq  \has{\int_{\ell_R}^{\ell_P} \int_{\R^n} \abs{f(y,t)}^2\cdot \hab{ \Phi\hab{\tfrac{\xi-y}{\alpha t}}-\Phi\hab{\tfrac{\eta-y}{\alpha t}}}^2 \frac{\ddn y\ddn t}{t^{n+1}}}^{1/2}\\
&\le\sum_{k=1}^N\has{\int_{2^{k-1}\ell_R}^{2^{k}\ell_R} \int_{\R^n}\abs{f(y,t)}^2\cdot \hab{ \Phi\hab{\tfrac{\xi-y}{\alpha t}}-\Phi\hab{\tfrac{\eta-y}{\alpha t}}}^2 \frac{\ddn y\ddn t}{t^{n+1}}}^{1/2}\\
&\lesssim \sum_{k=1}^N \frac{\abs{\xi-\eta}}{ \alpha 2^k \ell_R}\has{\int_{2^{k-1}\ell_R}^{2^{k}\ell_R} \int_{\abs{x-y} \leq 4 \alpha t+\sqrt{n} \ell_R}\abs{f(y,t)}^2 \frac{\ddn y\ddn t}{t^{n+1}}}^{1/2}\\
&\lesssim \sum_{k=1}^N \frac{1}{2^k}\has{\int_0^{\ell_P} \int_{\abs{x-y}<(4\alpha+\sqrt{n}) t}\abs{f(y,t)}^2 \frac{\ddn y\ddn t}{t^{n+1}}}^{1/2}\\
&\leq  A_{\ell_P}^{(4\alpha+\sqrt{n})}f(x).
\end{align*}
Therefore $m_P^{\#}f(x) \lesssim A_{\ell_P}^{(4\alpha+\sqrt{n})}f(x)$ for $x \in P$. Since we already noted that $f_P \leq A^{(2\alpha)}_{\ell^P}f$, we obtain by (\ref{rearch}),
\begin{align*}
 \gamma_P &\leq  \has{\frac{2^{n+3}}{\abs{P}} \int_P f_P(z)^2\dd z}^{1/2}+\has{\frac{2^{n+3}}{\abs{P}} \int_P m_P^{\#}f(z)^2\dd z}^{1/2} \\&\lesssim \has{\frac{1}{\abs{P}} \int_P A_{\ell_P}^{(4\alpha+\sqrt{n})}(f)(z)^2 \dd z}^{1/2}.
\end{align*}
Combined with \eqref{eq:sparsedomtent}, this finishes the proof.
\end{proof}

\subsection{An improved good-\texorpdfstring{$\la$}{lambda} inequality}
The estimate \eqref{eq:TunderC} was shown in \cite{CMS85} using the equivalence of tent spaces with different apertures and the following good-$\lambda$ estimate:  there exists a fixed $\a>1$
and a constant $c>0$ so that for all $\la>0$ and $0<\ga\le 1$,
\begin{equation}\begin{aligned}\label{eq:goodlambda}
  \absb{\{x\in \R^n:A(f)(x)>2\la, \,&C(f)(x)\le \ga \la\}}\\&\le c\,\ga^2 \absb{\{x \in \R^n:A^{(\a)}(f)(x)>\la\}}.
\end{aligned}
\end{equation}
where we abbreviated $A(f):=A^{(1)}(f)$.
Using Theorem \ref{theorem:sparsetent}, we can show that the quadratic dependence on $\ga$ in (\ref{eq:goodlambda}) can be improved to quadratic exponential dependence.

\begin{theorem}\label{theorem:exp}Let $f \colon \R^{n+1}_+ \to \R$ be measurable.
There exist constants $\a>1$ and $c>0$ so that for all $\la>0$ and $0<\ga\le 1$,
\begin{align*}
  \absb{\{x\in \R^n:A(f)(x)>2\la,\,& C(f)(x)\le \ga \la\}}\\&\le 2e^{-c/\ga^2}\absb{\{x\in \R^n:A^{(\a)}f(x)>\la\}}.
\end{align*}
\end{theorem}

Before proving this lemma, we establish the following simple proposition.

\begin{prop}\label{distr} Let $Q \in \mc{Q}$ and let ${\mathcal F}\subset {\mathcal D}(Q)$ be an $\eta$-sparse family. Then we have for any $\alpha>0$
$$\absb{\cbraceb{x\in Q:\sum_{P\in {\mathcal F}}\chi_{P}(x)>\a}}\le \frac{1}{1-\eta}e^{-(\log\frac{1}{1-\eta})\a}|Q|.$$
\end{prop}

\begin{proof} Write ${\mathcal F}=\cup_{k=0}^{\infty}{\mathcal F}_k$ as in the definition of a contracting family of cubes. By $\eta$-sparseness, we have $|\O_k|\le (1-\eta)^k|Q|$. Thus, it follows that
\begin{align*}
\absb{\cbraceb{x\in Q:\sum_{P\in {\mathcal F}}\chi_{P}(x)>\a}}&=\sum_{k=1}^{\infty}\abs{\O_{k-1}}\chi_{(k-1,k]}(\a)\\
&\le |Q|\sum_{k=1}^{\infty}(1-\eta)^{k-1}\chi_{(k-1,k]}(\a)\\&\le (1-\eta)^{\a-1}|Q|,
\end{align*}
which completes the proof.
\end{proof}

\begin{proof}[Proof of Theorem \ref{theorem:exp}]
As in \cite{CMS85}, we consider a Whitney decomposition $\cbrace{Q_j}_j$ of the open set $$\{x \in \R^n:A^{(\a)}(f)(x)>\la\},$$
where $\a>1$ will be chosen later on. Then it suffices to prove that for every $Q_j$,
\begin{equation}\label{localpart}
\absb{\{x\in Q_j:A(f)(x)>2\la, \,Cf(x)\le \ga \la\}}\le c_1e^{-c_2/\ga^2}|Q_j|.
\end{equation}

Define $f_{{Q_j}}(y,t):= f(y,t) \chi_{(\ell_{Q_j},\infty)}(t)$ and note that
$$A(f)\le A(f_{Q_j})+A_{\ell_{Q_j}}(f).$$
Consider $A(f_{Q_j})(x)$ for $x\in Q_j$.
By the properties of the Whitney cubes, there exist $z\in Q_j$ and $x_j\in \R^n$ such that $|z-x_j|\le 4\sqrt{n}\ell_{Q_j}$ and $A^{(\a)}(f)(x_j)\le \la$.
Hence, for $(y,t) \in \R^{n+1}_+$ with $|y-x|<t$ and $t\ge \ell_{Q_j}$, we obtain
\begin{align*}
|y-x_j|\le |y-x|+|x-z|+|z-x_j|
<t+5\sqrt{n}\ell_{Q_j} \le (5\sqrt{n}+1)t.
\end{align*}
Therefore, if $\a=5\sqrt{n}+1$, then $A(f_{Q_j})(x)\le A^{(\a)}(f)(x_j)\le \la$ for all $x\in Q_j$.
It follows that the left-hand side of (\ref{localpart}) is bounded by
\begin{equation}\label{Qj}
\absb{\{x\in Q_j:A_{\ell_{Q_j}}(f)(x)>\la,\, C(f)(x)\le \ga \la\}}.
\end{equation}

By Theorem \ref{theorem:sparsetent} and \eqref{eq:reformC}, there exists a $\frac{1}{2}$-sparse family ${\mathcal F}_j\subset {\mathcal D}(Q_j)$ such that
$$A_{\ell_{Q_j}}(f)(x)^2\lesssim C(f)(x)^2\sum_{P\in {\mathcal F}_j}\chi_P(x), \qquad x \in Q_j.$$
Combined with Proposition \ref{distr} this implies that the expression in (\ref{Qj}) is at most
$$
\absb{\{x\in Q_j:\sum_{P\in {\mathcal F}_j}\chi_P(x)\gtrsim 1/\ga^2\}}\le 2e^{-c/\ga^2}|Q_j|,
$$
which completes the proof of \eqref{localpart} and therefore of the theorem.
\end{proof}

\subsection{Vector-valued tent spaces}
Reinterpreting and extending the formulation of tent spaces by Harboure, Torrea and Viviani in \cite{HTV91},
Hyt\"onen, van Neerven and Portal \cite{HNP08} extended tent spaces to the the vector-valued setting. In this subsection we will point out how the arguments of the preceding subsection extend to this setting.

 In order to give the definition of these vector-valued tent spaces, we first need to introduce some notation. For a Banach space $X$ and a Hilbert space $H$, denote the space of $\gamma$-radonifying operators by $\gamma(H,X)\subseteq \mc{L}(X,H)$. For an introduction to these spaces we refer to \cite[Chapter~9]{HNVW17}.

For the remainder of this section, set $H:= L^2(\R^{n+1}_+,\frac{\ddn y\ddn t}{t^{n+1}})$. Then the space $\gamma(H,X)$ can be thought of as a square function space, since
 \begin{equation*}
  \gamma(H,L^p(\R^d)) = L^p(\R^d;L^2(\R^{n+1}_+,\tfrac{\ddn y\ddn t}{t^{n+1}})), \qquad p \in [1,\infty).
\end{equation*}
Let  $f \colon \R^{n+1}_+ \to X$ be strongly measurable. If $\ip{f,x^*} \in H$ for all $x^* \in X^*$, we can define the operator $I_f \in \mc{L}(H,X)$ by
\begin{equation*}
  I_f\varphi := \int_{\R^{n+1}_+} f(y,t) \varphi(y,t) \frac{\ddn y \ddn t}{t^{n+1}}, \qquad \varphi \in H,
\end{equation*}
where the integral is interpreted in the Pettis sense (see \cite[Theorem 1.2.37]{HNVW16}).
If $I_f \in \gamma(H,X)$, we write with slight abuse of notation $f \in \gamma(H,X)$ and $\nrm{f}_{\gamma(H,X)}:= \nrm{I_f}_{\gamma(H,X)}$. Moreover, if $I_f \notin \gamma(H,X)$ or $\ip{f,x^*} \notin H$ for some $x^* \in X^*$, we set $\nrm{f}_{\gamma(H,X)} = \infty$.

We are now ready to define the vector-valued tent spaces introduced in \cite{HNP08}. For $\alpha>0$ and a strongly measurable $f \colon \R^{n+1}_+ \to X$ define
\begin{align*}
  A^{(\alpha)}(f)(x)&:= \nrm{f\cdot \chi_{\Gamma_\alpha(x)}}_{\gamma(H,X)}, && x \in \R^n\\
  A^{(\alpha)}_h(f)(x)&:= \nrm{f\cdot \chi_{\Gamma_\alpha(x)} \cdot \chi_{\R^n \times (0,h)}}_{\gamma(H,X)}, &&x \in \R^n, \, h>0.
\end{align*}
 Since $\gamma(H,\R)=H$, this definition coincides with the scalar-valued definitions of $A^{(\alpha)}$ and $A^{(\alpha)}_h$. Therefore it makes sense to define $T^{p}_\alpha(X)$ as the completion of the space of all strongly measurable $f \colon \R^{n+1}_+ \to X$ such that
\begin{equation*}
  \nrm{f}_{T_\alpha^{p}(X)} := \nrm{A^{(\alpha)}(f)}_{L^p(\R^n)} <\infty.
\end{equation*}

It was shown in \cite[Theorem 4.3]{HNP08} that, as in the scalar case, $T_\alpha^{p}(X) = T_\beta^{p}(X)$ for $\alpha,\beta >0$ when $p \in (1,\infty)$ and $X$ has the so-called $\UMD$ property (see \cite[Chapter 4]{HNVW16}).

The scalar-valued definition of $C(f)$ does not make sense in the vector-valued setting. However, its reformulation using \eqref{eq:reformC} does. Following the work of Hyt\"onen and Weis \cite{HW10}, we slightly generalise this formulation. Fix $q \in (0,\infty)$, $\alpha>0$ and for a strongly measurable $f \colon \R^{n+1}_+ \to X$ define
\begin{equation*}
  C_q^{(\alpha)}(f)(x):=\sup_{B \ni x} \has{ \frac{1}{\abs{B}} \int_B A^{(\alpha)}_{r(B)}(f)^q }^{1/q},\qquad x \in \R^n.
\end{equation*}
If $X=\R$ and $q =2$, we have $C_2^{(\alpha)}(f) \eqsim C(f)$ by \eqref{eq:reformC}. 

The equivalence between $A^{(\alpha)}(f)$ and $C^{(\alpha)}(f)$ was proven in \cite[Theorem 4.4]{HW10}, using a vector-valued analogue of the good-$\lambda$ inequality \eqref{eq:goodlambda}.
Since this uses the equivalence of vector-valued tent spaces with different apertures, this result is limited to $p \in (1,\infty)$ and $\UMD$ Banach spaces.

As in the scalar-valued setting, we will give a ``sparse" proof the equivalence between $A^{(\alpha)}(f)$ and $C^{(\alpha)}(f)$. In the proof we will not use the equivalence of vector-valued tent spaces with different apertures, which allows us the treat $p \in (0,\infty)$ and arbitrary Banach spaces. The price we pay is that we have to increase the aperture of $C^{(\alpha)}(f)$. Of course, if $p \in (1,\infty)$ and $X$ has the $\UMD$ property, one can use the equivalence of vector-valued tent spaces with different apertures to recover \cite[Theorem 4.4]{HW10}.

We refer to \cite[Chapter 7]{HNVW17} for the definition of (Rademacher) type $r \in [1,2]$ with constant $\tau_{r,X}$ used in the following theorem. Let us note here that any Banach space has type $1$ with constant $\tau_{1,X}=1$.

\begin{theorem}\label{theorem:sparsetentvector} Let $X$ be a Banach space with type $r \in [1,2]$, take $q \in (0,\infty)$ and let $\alpha>0$. Let $f \colon \R^{n+1}_+ \to X$ be strongly measurable.  For every cube $Q\in \mc{Q}$ there exists a $\frac{1}{2}$-sparse family ${\mathcal F}\subset {\mathcal D}(Q)$ such that for a.e. $x \in Q$,
$$A_{\ell_Q}^{(\alpha)}(f)(x)\lesssim \tau_{r,X} \has{\sum_{P\in {\mathcal F}} \has{\frac{1}{\abs{P}} \int_P A_{\ell_P}^{(4\alpha+\sqrt{n})}(f)^q}^{r/q} \cdot \chi_P(x)}^{1/r}.$$
\end{theorem}

\begin{proof}
As in the proof of Theorem \ref{theorem:sparsetent}, let $\Phi$ be a smooth function such that
$\chi_{B(0,1)}\le \Phi\le \chi_{B(0,2)}$ and for $Q \in \mc{Q}$ and $P \in \mc{D}(Q)$ define
\begin{align*}
  f_Q(x)&:= \nrmb{(y,t) \mapsto  f(y,t) \cdot \Phi\hab{\tfrac{x-y}{\alpha t}} \cdot \chi_{(0,\ell_Q)}(t)}_{\gamma(H,X)}, && x \in Q,\\
  f_{P,Q}(x)&:= \nrmb{(y,t) \mapsto  f(y,t) \cdot \Phi\hab{\tfrac{x-y}{\alpha t}} \cdot \chi_{(\ell_P,\ell_Q)}(t)}_{\gamma(H,X)}, && x \in P.
\end{align*}
For the family $\{f_Q,f_{P,Q}\}_{Q \in \mc{Q}, P \in \mc{D}(Q)}$ the $\ell^r$-condition holds with $C_r = \tau_{r,X}$ by \cite[Proposition 9.4.13]{HNVW17}. The rest of the proof follows the lines of the proof of Theorem \ref{theorem:sparsetent}. The only alterations are the following:
\begin{itemize}
  \item We replace pointwise estimates by the fact that for $f \in \gamma(H,X)$ and $g \in L^\infty(\R^{n+1}_+)$ we have
\begin{equation*}
     \nrm{f \cdot g}_{\gamma(H,X)} \leq \nrm{g}_{L^\infty(\R^{n+1}_+)}\nrm{f}_{\gamma(H,X)}.
 \end{equation*}
 \item We use Theorem \ref{ofsdp} with $r=r$ instead of $r=2$.
 \item In the concluding estimate, we use (\ref{rearch}) for $q$ instead of $2$.\qedhere
\end{itemize}
\end{proof}

As in the scalar-valued setting, as a direct corollary of Theorem \ref{theorem:sparsetentvector}, we obtain the following: For $p,q \in (0,\infty)$, $\alpha>0$, a Banach space $X$ and any strongly measurable $f \colon \mathbb{R}^{n+1}_+ \to X$ we have
\begin{align*}
 \nrm{f}_{T^p_\alpha(X)} \lesssim \nrm{C^{(4\alpha+\sqrt{n})}_q(f)}_{L^p(\R^n)}, && p \in (0,\infty).
\end{align*}
Moreover, since $C_q^{(\alpha)}(f)^q \lesssim M(A^{(\alpha)}(f)^q)$, we have
\begin{align*}
 \nrm{C^{(\alpha)}_q(f)}_{L^p(\R^n)}&\lesssim  \nrm{f}_{T^p_{\alpha}(X)}, && 0<q<p<\infty.
\end{align*}
As noted before, this recovers  \cite[Theorem 4.4]{HW10} if $p \in (1,\infty)$ and $X$ is a $\UMD$ Banach space.

To conclude this subsection, let us note that, doing similar adaptations to the proof of Theorem \ref{theorem:exp} as we did in the proof of Theorem \ref{theorem:sparsetentvector}, we can improve the vector-valued good-$\lambda$ inequality in \cite[Theorem 4.4]{HW10} to exponential dependence on $\gamma^r$.

\section{Vector-valued square functions}\label{sec:squarefun}
In a recent paper by Xu \cite{Xu21}, vector-valued Littlewoood--Paley--Stein theory was developed using Littlewood--Paley theory and functional calculus methods, which vastly improves earlier approaches. In this section we will simplify the technical core of \cite{Xu21}, using our pointwise sparse domination principle.

 To introduce the main result of \cite{Xu21},  let $p \in (1,\infty)$, let $(\Omega,\mu)$ be a $\sigma$-finite measure space and let $X$ be a Banach space with martingale cotype $q \in [2,\infty)$ with constant $c_{q,X}^{\mart}$. We refer to  \cite[Section 3.5.d]{HNVW16} for an introduction to martingale (co)type.
 For a strongly continuous semigroup of regular operators $\cbrace{T_t}_{t\geq 0}$ on $L^p(\Omega)$ and its subordinated Poisson semigroup $\cbrace{P_t}_{t\geq 0}$, one of the main results of \cite{Xu21} states that for $f  \in L^p(\Omega;X)$ one has
 \begin{equation}\label{eq:Xuresult}
   \nrms{\has{\int_0^\infty \nrmb{t \tfrac{\partial}{\partial t}P_t(f)}_X^q}^{1/q}}_{L^p(\Omega)} \lesssim \max\cbrace{p^{\frac1q},p'} \cdot c_{q,X}^{\mart} \cdot \nrm{f}_{L^p(\Omega;X)}.
 \end{equation}

 The converse of this estimate is shown to hold under a martingale type assumption. Moreover, using functional calculus techniques, similar estimates with $\cbrace{T_t}_{t\geq 0}$ instead of $\cbrace{P_t}_{t\geq 0}$ are obtained under an analyticity assumption.
The growth order in $p$ in most of these estimates is sharp for $p \to 1$ and $p \to \infty$. When $\cbrace{T_t}_{t\geq 0}$ is the heat semigroup on $\R^n$, these results answer a question raised by Naor and Young in the appendix of \cite{NY20}.

The most technical part of the argument in \cite{Xu21} is a sharp estimate for a vector-valued variant of the vertical square function. For $\varepsilon,\delta>0$ let $\mc{H}_{\varepsilon,\delta}$ be the class of all $\varphi\colon \R^n \to \R$ such that $\int_{\R^n} \varphi = 0$ and
\begin{align}
\label{eq:phi1} \abs{\varphi(x)} &\leq \frac{1}{(1+\abs{x})^{n+\varepsilon}}, &&x \in \R^n \\
 \label{eq:phi2} \abs{\varphi(x)-\varphi(x')} &\leq \frac{\abs{x-x'}^\delta}{(1+\min\cbrace{\abs{x},\abs{x'}})^{n+\varepsilon+\delta}}, && x,x'\in \R^n.
  \end{align}
For $\varphi \in \mc{H}_{\varepsilon,\delta}$ and $f \in L^1(\R^n;X)$ define
\begin{equation*}
  G_{q,\varphi}(f)(x) = \has{\int_{0}^\infty\nrm{\varphi_t *f(x)}_X^q\frac{\ddn t}{t}}^{1/q}, \qquad x \in \R^n,
\end{equation*}
where $\varphi_t(x) = \frac{1}{t^n}\varphi\ha{\frac{x}{t}}$.
In \cite{Xu21} the main result \eqref{eq:Xuresult} follows from
\begin{equation}\label{eq:Xutechnical}
  \nrm{G_{q,\varphi}(f)}_{L^{p}(\R^n)} \lesssim \max\cbrace{p^{\frac1q},p'} \cdot c_{q,X}^{\mart} \,\nrm{f}_{L^p(\R^n;X)},
\end{equation}
by representing the left-hand side of \eqref{eq:Xuresult} for the Poisson semigroup subordinated to the translation group on $\R$ by $G_{q,\varphi}(f)$ for some $\varphi \in \mathcal{H}_{\frac12,1}$ and then using a transference argument for general semigroups.

The case $p < q$ of \eqref{eq:Xutechnical} follows quite easily from the case $p=q$, using classical Calder\'on--Zygmund theory. The case $p>q$ with optimal dependence on $p$ is harder, for which delicate results on conical and intrinsic square functions and weighted estimates, developed in the scalar-valued case by Wilson \cite{Wi07, Wi08}, are adapted to the vector-valued setting in \cite[Section 6]{Xu21}. We will prove \eqref{eq:Xutechnical} without the use of this machinery, instead opting to use our sparse domination principle. 

As a starting point we will use the following weak $L^1$-estimate for $G_{q,\varphi}$, which is implicitly contained in \cite{Xu21}.
\begin{prop}\label{prop:weakL1square}
  Let $q \in [2,\infty)$, let $X$ be a Banach space with martingale cotype $q$ and let $\varphi \in \mc{H}_{\varepsilon,\delta}$ for $\varepsilon,\delta>0$. Then we have for $f \in L^1(\R^n;X)$,
  \begin{equation*}
  \nrm{G_{q,\varphi}(f)}_{L^{1,\infty}(\R^n)} \lesssim c_{q,X}^{\mart} \,\nrm{f}_{L^1(\R^n;X)},
\end{equation*}
with the implicit constant only depending on $\varepsilon,\delta,n$.
\end{prop}
\begin{proof}
  The estimate
  \begin{equation*}
  \nrm{G_{q,\varphi}(f)}_{L^{q}(\R^n)} \lesssim c_{q,X}^{\mart} \,\nrm{f}_{L^q(\R^n;X)}
\end{equation*}
  follows directly from \cite[Lemma 5.6]{Xu21}, see the first half of Step 1 of the proof of \cite[Theorem 1.5]{Xu21}. The proposition then follows by viewing $G_{q,\varphi}$ as a Calder\'on--Zygmund operator using \cite[Lemma 5.4]{Xu21}.
\end{proof}

For $\varphi \in \mc{H}_{\varepsilon,\delta}$ for $\varepsilon,\delta>0$ define the localisation
\begin{equation*}
  G_{q,\varphi}^h(f)(x) = \has{\int_{0}^h\nrm{\varphi_t *f(x)}_X^q\frac{\ddn t}{t}}^{1/q}, \qquad x \in \R^n
\end{equation*}
for $h>0$. Since the support of $\varphi$ is not necessarily compact, the support of $G_{q,\varphi}^h(f)$ is not localised to (a multiple of) the support of $f$. Therefore, for arbitrary $f \in L^1(\R^d;X)$,  one can not estimate $G_{q,\varphi}^h(f)$ by a local expression of the form
  \begin{equation}\label{eq:sparsedomqnot}
\has{\sum_{P \in \mc{F}}  \ipb{\nrm{f}_X}_{1,\alpha P}^q \chi_P}^{1/q}, \qquad \alpha \geq 1.
  \end{equation}
  This, in particular, means that the precursor of Theorem \ref{ofsdp} in \cite{Lo19b}, i.e. Theorem \ref{ellrsdp}, is not applicable to the localisation $G_{q,\varphi}^{\ell_Q}(f)$.

   Thanks to the flexible formulation of Theorem \ref{ofsdp}, we are able to compensate the nonlocal behaviour of $G_{q,\varphi}^h(f)$ by adding a convergent series of dilations of $P$ to \eqref{eq:sparsedomqnot}. The main result of this section reads as follows:

\begin{theorem}\label{theorem:squarefunction}
  Let $q \in [2,\infty)$, let $X$ be a Banach space with martingale cotype $q$ and let $\varphi \in \mc{H}_{\varepsilon,\delta}$ with $\varepsilon,\delta>0$. For any $f \in L^1(\R^n;X)$ and $Q \in \mc{Q}$ there exists a $\frac{1}{2}$-sparse collection of cubes $\mc{F} \subset\mc{D}(Q)$ such that for a.e. $x \in Q$,
  \begin{equation*}
    G_{q,\varphi}^{\ell_Q}(f)(x) \lesssim  c_{q,X}^{\mart}  \has{\sum_{P \in \mc{F}} \sum_{m=1}^\infty \frac{1}{2^{m\varepsilon}}\ipb{\nrm{f}_X}_{1,2^{m}P}^q \chi_P(x)}^{1/q},
  \end{equation*}
  with the implicit constant depending only on $\varepsilon,\delta,n$.
\end{theorem}

Using H\"older's inequality and the boundedness of the maximal operator, Theorem  \ref{theorem:squarefunction} yields for $p>q$ and any $g \in L^{(p/q)'}(\R^n)$ that
\begin{align*}
\int_Q  G_{q,\varphi}^{\ell_Q}(f)^q g&\lesssim (c_{q,X}^{\mart})^q\sum_{P \in \mc{F}} \sum_{m=1}^\infty \frac{1}{2^{m\varepsilon}}\ipb{\nrm{f}_X}_{1,2^mP}^q \int_P g\\
&\lesssim  (c_{q,X}^{\mart})^q \sum_{P \in \mc{F}} \int_{E_P} M\hab{\nrm{f}_X }^q Mg\\
&\lesssim (c_{q,X}^{\mart})^q \cdot  \frac{p}{q} \cdot \nrm{f}_{L^p(\R^n;X)}^q \nrm{g}_{L^{(p/q)'}(\R^n)}.
\end{align*}
This yields \eqref{eq:Xutechnical} by duality and the monotone convergence theorem.

Moreover, one can deduce sharp weighted estimates for $G_{q,\varphi}(f)$ for weights in the Muckenhoupt $A_p$-class, using \cite[Lemma 4.5]{Le16} and an argument as in \cite[Section 4]{Le14}.

\begin{proof}[Proof of Theorem \ref{theorem:squarefunction}]
For $Q \in \mc{Q}$ and $P \in \mc{D}(Q)$ define
\begin{align*}
  f_Q(x) &:= G_{q,\varphi}^{\ell_Q}(f)(x),&& x \in \R^n,\\
  f_{P,Q}(x) &:= \has{\int_{\ell_P}^{\ell_Q}\nrm{\varphi_t *f(x)}_X^q\frac{\ddn t}{t}}^{1/q},&& x \in \R^n.
\end{align*}
The family $\{f_Q,f_{P,Q}\}_{Q \in \mc{Q}, P \in \mc{D}(Q)}$ trivially satisfies the $\ell^q$-condition with $C_q=1$, and condition (\ref{cond}) holds as well.
Therefore, by Theorem \ref{ofsdp}, there exists a $\frac{1}{2}$-sparse family ${\mathcal F}\subset {\mathcal D}(Q)$ such that
\begin{equation*}
 G_{q,\varphi}^{\ell_Q}(f)(x)\lesssim \has{\sum_{P\in {\mathcal F}}\ga_P^q\chi_P(x)}^{1/q},\qquad x\in Q.
\end{equation*}
Thus it suffices to show
\begin{equation}\label{eq:gammaforsquare}
  \ga_P \lesssim c_{q,X}^{\mart} \cdot \has{ \sum_{m=1}^\infty \frac{1}{2^{m\varepsilon}}\ipb{\nrm{f}_X}_{1,2^mP}^q}^{1/q}:= c_{q,X}^{\mart} \cdot \mc{M}_P
\end{equation}
for $P \in \mc{F}$.

Fix $P \in \mc{F}$. For any $z \in P$ we have by \eqref{eq:phi1} and H\"older's inequality
\begin{align} \label{eq:nonlocalpart}
\begin{aligned}
  G_{q,\varphi}^{\ell_P}&(f\chi_{\R^n \setminus 2P})(z)\\ 
  &\leq \has{\int_{0}^{\ell_P} \has{\sum_{m=1}^\infty\int_{(2^{m+1}P)\setminus (2^mP)} \frac{1}{{\abs{z-y}}^{d+\varepsilon}} \nrm{f(y)}_X  \dd y}^q\frac{\ddn t}{t^{1-q\varepsilon}}}^{\frac1q}\\
  &\lesssim \sum_{m=2}^\infty \frac{1}{2^{m\varepsilon}}\ipb{\nrm{f}_X}_{1,2^mQ}\cdot \has{\ell_P^{-q\varepsilon} \int_{0}^{\ell_P}  \frac{\ddn t}{t^{1-q\varepsilon}}}^{\frac1q}
   \lesssim  \mc{M}_P.
\end{aligned}
\end{align}
Therefore, we have by the weak $L^1$-boundedness of $G_{q,\varphi}$ in Proposition~\ref{prop:weakL1square}, that
\begin{align*}
  (f_P\chi_P)^*(|P|/2^{n+3}) &\lesssim \hab{G_{q,\varphi}(f\chi_{2P})\chi_P}^*(|P|/2^{n+3}) + \mc{M}_P\\
  &\lesssim  c_{q,X}^{\mart} \cdot \ipb{\nrm{f}_X}_{1,2P} +\mc{M}_P \lesssim \,c_{q,X}^{\mart} \cdot \mc{M}_P.
\end{align*}

Now let us turn to $(m^{\#}_Pf)^*(|P|/2^{n+3})$. Fix $x \in P$ and $R \in \mc{D}(P)$ such that $x \in R$.
We will split \begin{equation}\label{eq:splittingf}
  f = f\chi_{\R^n\setminus 2P} + f\chi_{2P\setminus 2R}+ f\chi_{2R}.
\end{equation}
For $\xi,\eta\in R$ we note that by \eqref{eq:phi2} and \cite[Theorem 2.1.10]{Gr14a} we have
\begin{align*}
  \has{\int_{\ell_R}^{\ell_P}&\nrmb{\varphi_t *f\chi_{2P\setminus 2R}(\xi)-\varphi_t *f\chi_{2P\setminus 2R}(\eta)}_X^q\frac{\ddn t}{t}}^{1/q}\\
  &\lesssim \has{\int_{\ell_R}^{\ell_P} \has{ \abs{\xi-\eta}^{\delta} \int_{\R^n\setminus B(x,{\ell_R/2})}\frac{\nrm{f(y)\chi_{2P}(y)}_X}{\abs{x-y}^{n+\frac{\delta}{2}}}\dd y}^q \frac{\ddn t}{t^{1+\frac{\delta}{2}}}}^{1/q}\\
  &\lesssim \ell_R^\delta \cdot \nrms{\frac{\chi_{\R\setminus B(0,\ell_R/2)}}{\abs{\cdot}^{n+\frac{\delta}{2}}}}_{L^1(\R^n)} \cdot \int_{\ell_R}^{\ell_P} \frac{\ddn t}{t^{1+\frac{\delta}{2}}}\cdot
  M\hab{\nrm{f\chi_{2P}}_X }(x)\\
  &\lesssim M\hab{\nrm{f\chi_{2P}}_X }(x).
\end{align*}
Furthermore, by \eqref{eq:phi1} we have for $\xi \in R$
\begin{align*}
  \has{\int_{\ell_R}^{\ell_P}\nrmb{\varphi_t *f\chi_{2R}(\xi)}_X^q\frac{\ddn t}{t}}^{1/q} &\leq \int_{2R} \nrm{f(y)}_X \dd y\cdot  \has{\int_{\ell_R}^{\ell_P}\frac{\ddn t}{t^{1+qn}}}^{1/q}\\
  &\lesssim M\hab{\nrm{f\chi_{2P}}_X }(x).
\end{align*}

Splitting as in \eqref{eq:splittingf}, combining these estimates with \eqref{eq:nonlocalpart} and using the weak $L^1$-boundedness of the maximal operator, we therefore obtain
\begin{align*}
  (m^{\#}_Pf)^*(|P|/2^{n+3})&\lesssim \hab{M(\nrm{f\chi_{2P}}_X)}^*(|P|/2^{n+3}) +\mc{M}_P\\  &\lesssim  \ipb{\nrm{f}_X}_{1,2P} +\mc{M}_P \lesssim \mc{M}_P .
\end{align*}
This finishes the proof of \eqref{eq:gammaforsquare} and thus the proof of the theorem.
\end{proof}

\section{An application to dyadic sums}\label{sec:dyadic}
In this final section we will give a condition on a sequence $\{\a_R\}_{R\in {\mathcal D}(Q)}$ for $Q \in \mc{Q}$ that allows to control a dyadic sum of the form $\sum_{R\in {\mathcal D}(Q)}\a_R\chi_R$ by a sum over a sparse family ${\mathcal F}\subset {\mathcal D}(Q)$. As an application, we will generalise and provide a new proof of a good-$\lambda$ inequality  of Honz\'ik and Jaye \cite{HJ12}.

\begin{theorem}\label{sumssparse} Let $Q\in {\mathcal Q}$ and let $\{\a_R\}_{R\in {\mathcal D}(Q)}$ be a sequence of nonnegative numbers. Suppose that there exist
$C>0$ and $0<\d\le 1$ such that for every cube $Q'\in {\mathcal D}(Q)$,
\begin{equation}\label{dsmall}
\sum_{R\in {\mathcal D}(Q')}\a_R^{\d}|R|\le C\a_{Q'}^{\d}|Q'|.
\end{equation}
Then there exists a $\frac{1}{2}$-sparse family ${\mathcal F}\subset {\mathcal D}(Q)$ such that for a.e. $x\in Q$,
$$
\sum_{R\in {\mathcal D}(Q)}\a_R\chi_R(x)\lesssim C \,\sum_{P\in {\mathcal F}}\a_P\chi_P(x).
$$
\end{theorem}

\begin{proof} For $Q' \in \mc{D}(Q)$ denote
$$f_{Q'}(x)=\sum_{R\in {\mathcal D}(Q')}\a_R\chi_R(x), \qquad x \in Q',$$
and for $P\in  \mc{D}(Q')$ set $f_{P,Q'}=f_{Q'}-f_P$. Then $\{f_{Q'},f_{P,Q'}\}_{Q' \in \mc{D}(Q), P \in \mc{D}(Q')}$ trivially satisfies the $\ell^1$-condition (with $C_1=1$) and condition (\ref{cond}).

Observe that $f_{P,Q}$ is a constant on~$P$, and therefore $m_P^{\#}f\equiv 0$. Hence, by the local version of Theorem \ref{ofsdp} (see Remark \ref{localcase}),
there exists a $\frac{1}{2}$-sparse family ${\mathcal F}\subset {\mathcal D}(Q)$ such that, for a.e. $x\in Q$,
$$
\sum_{R\in {\mathcal D}(Q)}\la_R\chi_R(x)\le \sum_{P\in {\mathcal F}}(f_P\chi_P)^*(|P|/2^{n+3})\chi_P(x).
$$
By (\ref{rearch}),
$$
(f_P\chi_P)^*(|P|/2^{n+3})\le \has{\frac{2^{n+3}}{|P|}\int_Pf_P^{\d}}^{1/\d}
$$
and, by (\ref{dsmall}),
$$
\int_Pf_P^{\d}\le \sum_{R\in {\mathcal D}(P)}\a_R^{\d}|R|\le C\a_P^{\d}|P|.
$$
Combining these three estimates  completes the proof.
\end{proof}

Let ${\mathscr D}$ be a dyadic lattice in ${\mathbb R}^n$.
Given a sequence of nonnegative numbers $\mbs{\alpha}=\{\a_Q\}_{Q\in {\mathscr D}}$,
define the following
two objects associated with $\mbs{\a}$:
\begin{align*}
  S_q(\mbs{\a})&:=\Big(\sum_{Q\in {\mathscr D}}\a_Q^q\chi_Q\Big)^{1/q}\qquad q \in (0,\infty),\\
M({\mbs{\a}})&:=\sup_{Q\in {\mathscr D}}\a_Q\chi_Q.
\end{align*}

\begin{cor}\label{goodlam}
Let $q \in (0,\infty)$. Suppose that there exist $C>0$ and $0<\d\le q$ such that for every cube $Q\in {\mathscr D}$,
\begin{equation}\label{dqsmall}
\sum_{R\in {\mathcal D}(Q)}\a_R^{\d}|R|\le C\a_Q^{\d}|Q|.
\end{equation}
Then there exists $K=K(q,\d,C)$ such that for all $\la>0$ and $0<\e<1$,
\begin{equation}\label{goodsq}
\begin{aligned}
\absb{\{x \in \R^n:S_q(\mbs{\a})(x)>&2\la, M(\mbs{\a})(x)\le \e\la\}}
\\&\le 2e^{-K/\e^q}\absb{\{x\in \R^n:S_q(\mbs{\a})(x)>\la\}}.
\end{aligned}
\end{equation}
\end{cor}

\begin{proof}
By a standard limiting argument, it sufficed to prove (\ref{goodsq}) for $S^F_q(A)$ instead of $S_q(A)$,
where
$$S_q^F(\mbs{\a})=\Big(\sum_{Q\in F}\a_Q^q\chi_Q\Big)^{1/q}$$
for  an arbitrary finite family of cubes $F\subset {\mathscr D}$.

Write the set $\{x \in \R^n:S_q^F(\mbs{\a})(x)>\la\}$
as the union of its maximal cubes $Q_j\in F$. Then it suffices to prove that
\begin{equation}\label{locpart}
\absb{\{x\in Q_j:S_q^F(\mbs{\a})(x)>2\la, M(\mbs{\a})(x)\le \e\la\}}\le 2e^{-K/\e^q}|Q_j|.
\end{equation}

Denote the set on the left-hand side of (\ref{locpart}) by $E_j$ and fix $x\in E_j$. By maximality of $Q_j$,
$$\sum_{R\in F: Q_j\subset R}\a_R^q\chi_R(x)\le \la^q,$$
and hence
$$\sum_{R\in F:R\subseteq Q_j}\a_R^q\chi_R(x)=S_q^F(\mbs{\a})(x)^q-\sum_{R\in F, Q_j\subset R}\a_R^q\chi_R(x)>(2^q-1)\la^q.$$
On the other hand, applying Theorem \ref{sumssparse} to $\cbrace{\alpha_R^q}_{R \in \mc{D}(Q_j)}$, there exists a $\frac{1}{2}$-sparse family ${\mathcal F}\subset {\mathcal D}(Q_j)$ such that for a.e. $x\in E_j$
$$
\sum_{R\in F:R\subseteq Q_j}\a_R^q\chi_R(x)\lesssim \sum_{P\in {\mathcal F}}\a_P^q\chi_P(x)
\lesssim (\e\la)^q\sum_{P\in {\mathcal F}}\chi_P(x).
$$
So we have $\sum_{P\in {\mathcal F}}\chi_P(x) \gtrsim \frac{1}{\varepsilon^q}$ and therefore, by Proposition \ref{distr},
$$|E_j|\le \Big|\Big\{x\in Q_j:\sum_{P\in {\mathcal F}}\chi_P(x)\gtrsim \frac{1}{\e^q}\Big\}\Big|\le 2e^{-K/\e^q}|Q_j|,$$
i.e. \eqref{locpart} holds and the proof is complete.
\end{proof}

\begin{example}\label{dpcalc}
Let $\mu$ be a nonnegative Borel measure. Given $0<\ga<n$ and $q \in (0,\infty)$, define
the nonlinear dyadic potential by
\begin{align*}
  {\mathcal T}_{q,\gamma}(\mu)&:=\has{\sum_{Q\in {\mathscr D}}\Big(\frac{\mu(Q)}{|Q|^{1-\ga/n}}\Big)^q\chi_Q}^{1/q}\intertext{
Define also the fractional maximal operator by}
M_{\ga}(\mu)&:=\sup_{Q\in {\mathscr D}}\frac{\mu(Q)}{|Q|^{1-\ga/n}}\chi_Q.
\end{align*}

In \cite{HJ12}, Honz\'ik and Jaye established the following good-$\lambda$ inequality: there exists $C_1,C_2>0$ such that for all $\la>0$ and $0<\e<1$,
\begin{equation}\label{hj}
\begin{aligned}
\absb{\{x \in \R^n:{\mathcal T}_q(\mu)(x)&>2\la, M_{\ga}(\mu)(x)\le \e\la\}}
\\&\le C_1e^{-C_2/\e^q}\absb{\{x \in \R^n:{\mathcal T}_q(\mu)(x)>\la\}}.
\end{aligned}
\end{equation}

Let us show that this result can be deduced from Corollary \ref{goodlam}. Indeed, set $\a_Q=\frac{\mu(Q)}{|Q|^{1-\ga/n}}$ for $Q \in \ms{D}$.
It suffices to show that (\ref{dqsmall}) holds for $\d=\min(q,1)$.
Write ${\mathcal D}(Q)=\cup_{k=0}^{\infty}{\mathcal D}_k$, where ${\mathcal D}_k$ is the $k$th generation of dyadic subcubes of $Q$. First suppose that $q\ge 1$. Then $\d=1$ and we have
\begin{align*}
\sum_{R\in {\mathcal D}(Q)}\a_R^{\delta}|R|=\sum_{R\in {\mathcal D}(Q)}\mu(R)|R|^{\ga/n}&=|Q|^{\ga/n}\sum_{k=0}^{\infty}2^{-k\ga}\sum_{P\in {\mathcal D}_k}\mu(P)\\
&=C_{\ga}\cdot \mu(Q)|Q|^{\ga/n}=C_{\ga}\cdot \a_Q^{\delta}|Q|.
\end{align*}
Now suppose that $q<1$. Since $\#\{Q\in {\mathcal D}_k\}=2^{kn}$, we have by H\"older's inequality
$$\sum_{P\in {\mathcal D}_k}\mu(P)^q\le 2^{nk(1-q)}\mu(Q)^q.$$
Hence, as $\d=q$,
\begin{align*}
\sum_{R\in {\mathcal D}(Q)}\a_R^{\d}|R|&=|Q|^{1-q(1-\ga/n)}\sum_{k=0}^{\infty}2^{nk((1-\ga/n)q-1)}\sum_{P\in {\mathcal D}_k}\mu(P)^q\\
&\le \Big(\frac{\mu(Q)}{|Q|^{1-\ga/n}}\Big)^q|Q|\sum_{k=0}^{\infty}2^{nk((1-\ga/n)q-1)}2^{nk(1-q)}\\
&= \a_Q^q|Q|\sum_{k=0}^{\infty}2^{-k\ga q}=C_{\ga,q}\cdot \a_Q^\delta|Q|.
\end{align*}
Thus, we have verified (\ref{dqsmall}) and therefore (\ref{hj}) holds.
\end{example}

\end{document}